\documentclass[11pt,oneside]{article}
\usepackage[square, numbers]{natbib}
\usepackage[a4paper,margin=1in]{geometry}
\usepackage{amsmath,amssymb,amsthm,mathtools}
\usepackage{microtype}
\usepackage{hyperref}
\hypersetup{hidelinks}

\numberwithin{equation}{section}

\newtheorem{theorem}{Theorem}[section]
\newtheorem{proposition}[theorem]{Proposition}
\newtheorem{lemma}[theorem]{Lemma}
\newtheorem{corollary}[theorem]{Corollary}
\newtheorem{definition}[theorem]{Definition}
\newtheorem{example}[theorem]{Example}
\theoremstyle{remark}
\newtheorem{remark}[theorem]{Remark}

\newcommand{\R}{\mathbb{R}}
\newcommand{\T}{\mathbb{T}}
\newcommand{\osc}{\operatorname{osc}}
\newcommand{\cvar}{\mathrm{var}}

\title{Partition invariance of variation indices \\ and classical $p$-variation spaces}

\author{
	\textsc{Donghan Kim} 
	\thanks{Department of Mathematical Sciences, KAIST, South Korea (E-mail: {\it kimdonghan@kaist.ac.kr})}
}

\date{\today}

\begin{document}

\maketitle

\begin{abstract}
    We study the variation index of a continuous path along a refining partition sequence, defined as the infimum of exponents $p\ge1$ for which the corresponding $p$-th variation sums are uniformly bounded. Under natural geometric assumptions on the partitions, we prove that this index coincides with the classical variation index, defined using all finite partitions, for every path with positive H\"older regularity. The result extends to paths with an all-orders Dini modulus and yields a characterization of the classical variation index in terms of dyadic Faber--Schauder coefficients. Explicit counterexamples show that the partition assumptions cannot in general be omitted. Even under these assumptions, however, membership in the corresponding function spaces at the critical exponent may still depend on the partition sequence. We further obtain compact subcritical embeddings into the space of paths with vanishing classical $p$-variation and establish a strict hierarchy of coefficient and variation spaces. Finally, we show that the continuous embedding of the classical $p$-variation space into the space of continuous paths with uniformly bounded dyadic $p$-th variation sums has nonclosed range.
\end{abstract}

\medskip

\noindent\textit{MSC 2020:} 26A16, 26A45, 40A05, 46E15, 60L20.\\
\noindent\textit{Keywords:} variation index, classical $p$-variation, refining partition sequence, H\"older regularity, Faber--Schauder coefficients, compact embedding.

\bigskip

\section{Introduction}\label{sec:introduction}

The classical $p$-variation of a continuous path is an intrinsic measure of its oscillation, obtained by taking the supremum of the corresponding variation sums over all finite partitions of the time interval. Spaces of paths with finite classical $p$-variation, together with the subspaces obtained by closing piecewise linear paths in the $p$-variation norm, play a fundamental role in deterministic integration and rough path theory; see, for instance, \cite{friz2010,FrizHairer}. In pathwise constructions based on discrete observations, however, one often fixes a refining sequence of partitions $\pi=(\pi^n)_{n\ge0}$ with vanishing mesh and considers only the increments of the path along the prescribed partition points. This leads to the following variation index of a path $x\in C^0([0,T])$ along $\pi$:
\[
    p_\pi(x):=\inf\Big\{q\ge1:\sup_{n\ge0}[x]_{\pi^n}^{(q)}(T)<\infty\Big\},
\]
which is the critical threshold for the uniform boundedness of the variation sums along $\pi$. By contrast, the classical variation index is
\[
    p_{\cvar}(x):=\inf\Big\{q\ge1:\|x\|_{q\text{-}\cvar}<\infty\Big\},
\]
where the classical $q$-variation seminorm is defined using all finite partitions. Unlike $p_{\cvar}(x)$, the quantity $p_\pi(x)$ may depend on the sampling sequence $\pi$.

The use of variation along a prescribed partition sequence in pathwise calculus originates in F\"ollmer's pathwise It\^o formula, which defines integrals of gradients of smooth functions along paths with finite quadratic variation and proves a deterministic change-of-variable formula \cite{follmer1981}. In the same quadratic-variation setting, a functional-analytic theory of the resulting integral was developed, including a pathwise It\^o isometry, continuity on a suitable integrand space, and a rough-smooth decomposition \cite{ananova2017}. Higher-order extensions include change-of-variable formulas based on even-order $p$-th variation \cite{perkowski2019} and fractional formulas for arbitrary $p>1$ \cite{ruhong2022}. The dependence of quadratic variation and related roughness properties on the chosen refining sequence has been investigated in \cite{das2021,das2020}. These developments motivate the study of the critical exponent associated with a prescribed partition sequence, even when the corresponding variation sums do not converge.

Beyond its role in pathwise calculus, the variation index provides a notion of pathwise roughness that complements H\"older regularity. For fractional Brownian motion, the dyadic variation index almost surely agrees with the reciprocal of the critical H\"older exponent, but this identity need not hold even for paths with positive H\"older regularity; explicit examples were given in \cite{fake_fBM}. Thus, the critical H\"older exponent alone does not determine the value of the variation index. This raises a different question: can positive H\"older regularity nevertheless ensure that the variation index is independent of the choice among suitably regular partition sequences? The present paper addresses this question without requiring the variation index to equal the reciprocal of the critical H\"older exponent.

In a recent paper \cite{das-kim2024}, the space
\[
    \mathcal X_\pi^p:=\Big\{x\in C^0([0,T]):\sup_{n\ge0}[x]_{\pi^n}^{(p)}(T)<\infty\Big\}
\]
was studied for a general class of refining partition sequences. Membership in $\mathcal X_\pi^p$ was characterized in terms of the Schauder coefficients associated with $\pi$. Under additional H\"older regularity assumptions, a Ciesielski-type isomorphism with an appropriate coefficient space was also established, extending the classical result of Ciesielski \cite{Ciesielski:isomorphism}. Complementary recent work studies the more restrictive class in which the $p$-th variation sums converge along a fixed partition sequence, and develops a constructive and functional-analytic framework for paths with prescribed $p$-th variation, including Banach subspaces with explicitly controlled variation \cite{das-kim-lim2026}. The present paper addresses a different question. Rather than fixing $p$ and characterizing the resulting space $\mathcal X_\pi^p$, we ask when its critical threshold
\[
    p_\pi(x)=\inf\{q\ge1:x\in\mathcal X_\pi^q\}=\inf\Big\{q\ge1:\sup_{n\ge0}[x]_{\pi^n}^{(q)}(T)<\infty\Big\}
\]
is independent of the partition sequence. A central phenomenon revealed below is that this critical threshold is considerably more stable than membership at the critical exponent itself.

Since each $\pi^n$ is one of the finite partitions appearing in the definition of the classical variation, one always has $p_\pi(x)\le p_{\cvar}(x)$. Our principal result identifies conditions under which the reverse inequality holds. More precisely, under two assumptions on the partition sequence, namely the balancedness condition in \eqref{eq:balanced-intro} and the upper bound on the successive mesh ratios in \eqref{eq:no-large-jump-intro}, Theorem~\ref{thm:main-intro} shows that
\begin{equation}\label{eq:introduction-main-result}
    p_\pi(x)=p_{\cvar}(x)=p_\T(x)
\end{equation}
for every path $x$ with positive H\"older regularity, where $\T$ denotes the dyadic partition sequence. The class of partition sequences covered by the theorem contains every balanced complete refining partition sequence considered in \cite{das-kim2024}, but is strictly larger, since no lower bound strictly larger than one is imposed on the successive mesh ratios. By equivalence of norms, the same conclusion extends immediately from real-valued continuous paths to finite-dimensional ones.

We further show that positive H\"older regularity is not essential to the proof of \eqref{eq:introduction-main-result}; Corollary~\ref{cor:dini-main} replaces it by an all-orders Dini condition on the modulus of continuity. A stretched-logarithmic function in Example~\ref{ex:non-holder-modulus} satisfies this condition without belonging to $C^\alpha([0,T])$ for any $\alpha>0$, so this extension genuinely goes beyond the union of the positive H\"older classes.

Separate counterexamples show that neither the balancedness condition nor the upper bound on the successive mesh ratio can in general be removed. If the upper bound on the mesh ratio is omitted, a perfectly balanced sparse subsequence of the dyadic partitions may have variation index one even when the dyadic variation index is any prescribed number $s>1$. If the balancedness condition is omitted, an analogous failure may occur even when the largest mesh decreases exactly by a factor of two at every level. These examples distinguish the roles of scale regularity and spatial balance in the partition-invariance result \eqref{eq:introduction-main-result}.

Equality of variation indices does not imply equality of the corresponding spaces at the critical exponent. Suppose that $p_\pi(x)=p_\sigma(x)=p_{\cvar}(x)=p<\infty$. Then membership in $\mathcal X_\pi^q$ and $\mathcal X_\sigma^q$ necessarily agrees for every noncritical exponent $q\ne p$, whereas no such conclusion follows at $q=p$. Proposition~\ref{prop:critical-space-noninvariance} shows that this limitation is genuine. For every $p>1$, we construct a path and two balanced complete refining partition sequences $\pi$ and $\sigma$ such that
\[
    x\in\bigcap_{0<\alpha<1/p}C^\alpha([0,1]),\qquad p_\pi(x)=p_\sigma(x)=p_{\cvar}(x)=p,\qquad x\in\mathcal X_\sigma^p\setminus\mathcal X_\pi^p.
\]
Thus, under the balancedness and mesh-ratio assumptions, the variation index becomes canonical, while membership in the corresponding critical space may remain sensitive to the locations of the partition points.

The second part of the paper develops the consequences of this distinction for classical $p$-variation spaces. Let $\mathcal C^{p\text{-}\cvar}([0,1])$ denote the Banach space of continuous paths with finite classical $p$-variation, and let $\mathcal C^{0,p\text{-}\cvar}([0,1])$ be the closure of the continuous piecewise linear paths in the classical $p$-variation norm. Under a related assumption on $\pi$, Proposition~\ref{prop:subcritical-compact} shows that, for $\alpha\in(0,1]$ and $1<q<p$, the embedding
\begin{equation}\label{eq:introduction-compact-embedding}
    C^\alpha([0,1])\cap\mathcal X_\pi^q\hookrightarrow\mathcal C^{0,p\text{-}\cvar}([0,1])
\end{equation}
is compact. This shows that subcritical variation control along one sufficiently regular partition sequence, together with positive H\"older regularity, yields approximation by piecewise linear paths in the stronger classical $p$-variation topology.

To study the critical exponent more explicitly, we specialize to the dyadic Faber--Schauder system. If $\theta_{n,k}^{x,\T}$ denotes the Faber--Schauder coefficient of $x$ at level $n$ and position $k$, set
\[
    \xi_n^{(p)}(x):=2^{-\frac{np}{2}}\sum_{k\in I_n}|\theta_{n,k}^{x,\T}|^p,\qquad I_n:=\{0,1,\ldots,2^n-1\}.
\]
The results of \cite{das-kim2024} identify uniform boundedness of these levelwise quantities with membership in $\mathcal X_\T^p$. Combining this characterization with the all-orders Dini extension of the partition-invariance theorem gives
\[
    p_{\cvar}(x)=\inf\Big\{q>1:\limsup_{n\to\infty}\xi_n^{(q)}(x)<\infty\Big\}
\]
for every path with an all-orders Dini modulus. Thus, although the classical variation index is defined using all finite partitions, it can be recovered from one canonical family of dyadic Faber--Schauder coefficients.

At a fixed exponent $p>1$, define
\[
    \mathcal S_\T^p:=\Big\{x\in C^0([0,1]):\sum_{n=0}^\infty\big(\xi_n^{(p)}(x)\big)^{1/p}<\infty\Big\}.
\]
We prove the inclusion chain
\begin{equation}\label{eq:introduction-coefficient-chain}
    \mathcal S_\T^p\subset\mathcal C^{0,p\text{-}\cvar}([0,1])\subset\mathcal C^{p\text{-}\cvar}([0,1])\subset\mathcal X_\T^p.
\end{equation}
The spaces $\mathcal S_\T^p$ and $\mathcal X_\T^p$ admit explicit descriptions in terms of the dyadic Faber--Schauder coefficients. Moreover, every positively H\"older continuous path in $\mathcal X_\T^q$ for some $q<p$ belongs to $\mathcal S_\T^p$. More generally, the positive H\"older regularity assumption can be replaced by an all-orders Dini condition. Hence, within these regularity classes, the discrepancy between the coefficient conditions and finite classical $p$-variation is confined to the critical exponent. Nevertheless, Theorem~\ref{thm:strict-endpoint-hierarchy} shows, by explicit examples, that every inclusion in \eqref{eq:introduction-coefficient-chain} is strict.

Finally, Corollary~\ref{cor:nonclosed-classical-embedding} shows that the continuous embedding
\[
    \mathcal C^{p\text{-}\cvar}([0,1])\hookrightarrow\mathcal X_\T^p
\]
does not have closed range. Thus, the dyadic norm on $\mathcal X_\T^p$ induces a genuinely weaker topology on the classical $p$-variation space and is not equivalent to the classical $p$-variation norm.

The remainder of the paper is organized as follows. Section~\ref{sec:partition-invariance} proves the partition-invariance theorem, extends it beyond H\"older regularity, establishes the sharpness of the partition assumptions, and constructs the critical-space counterexample. Section~\ref{sec:classical-pvar-spaces} studies the resulting relations with classical $p$-variation spaces, proves the compact subcritical embedding, derives the Faber--Schauder characterizations, and establishes the strict hierarchy in \eqref{eq:introduction-coefficient-chain} and the nonclosed-range result.

\bigskip

\section{Partition invariance of variation indices} \label{sec:partition-invariance}

\subsection{Notations}  \label{subsec:notations}

For a fixed $T>0$, let $C^0([0, T])$ denote the space of real-valued continuous functions defined on the finite interval $[0, T]$. We denote by $C^{\alpha}([0,T])$ the space of $\alpha$-H\"older continuous functions with the $\alpha$-H\"older seminorm
\[
    |x|_{C^{\alpha}} := \sup_{\substack{s, t \in [0, T] \\ s \neq t}} \frac{|x(t)-x(s)|}{|t-s|^{\alpha}}.
\]

Let $\pi=(\pi^n)_{n\ge 0}$ be a sequence of partitions of $[0,T]$ such that
\[
    \pi^n=\{ 0 = t_0^n < t_1^n < \cdots < t_{N(\pi^n)}^n = T\}, \qquad n \ge 0.
\]
We say that $\pi$ is \emph{refining} if $\pi^n \subset \pi^{n+1}$ for all $n \ge 0$.
With the notation
\begin{equation}\label{eq:mesh-notation}
    \underline{\pi^n} :=\min_{0\le j<N(\pi^n)}(t_{j+1}^n-t_j^n), \qquad  |\pi^n| :=\max_{0\le j<N(\pi^n)}(t_{j+1}^n-t_j^n),
\end{equation}
we say that $\pi$ has \emph{vanishing mesh} if $|\pi^n| \to 0$ as $n \to \infty$. Throughout the paper, $\Pi([0,T])$ denotes the collection of refining partition sequences of $[0,T]$ with vanishing mesh. We use the convention $\pi^0=\{0,T\}$.

We say that $\pi$ is \emph{balanced} if there exists $c>1$ such that
\begin{equation}\label{eq:balanced-intro}
    \frac{|\pi^n|}{\underline{\pi^n}}\le c \quad \text{holds for every} \quad n\ge 0.
\end{equation}
We also say that $\pi$ is \emph{perfectly balanced} if $|\pi^n| = \underline{\pi^n}$ for all $n \ge 0$. In particular, when $\pi^n$ satisfies for all $n\ge 0$
\[
    N(\pi^n) = 2^n, \qquad t^n_k = \frac{kT}{2^n}, \qquad k = 0, 1, \ldots, 2^n,
\]
we write this sequence $\T$ and call $\T = (\T^n)_{n \ge 0}$ the dyadic partition sequence of $[0, T]$.

For $p\ge 1$ and $x\in C^0([0,T])$, we define the $p$-th variation of $x$ along $\pi^n$
\begin{equation}\label{eq:intro-pvar}
    [x]_{\pi^n}^{(p)}(T) := \sum_{j = 0}^{N(\pi^n)-1} \big|x(t_{j+1}^n)-x(t_j^n)\big|^p, \qquad n \ge 0,
\end{equation}
and the variation index of $x$ along $\pi$
\begin{equation}\label{eq:intro-index}
    p_\pi(x) := \inf \big\{p\ge 1 : \sup_{n\ge 0} [x]^{(p)}_{\pi^n}(T) < \infty \big\}.
\end{equation}
The variation index, as defined in \cite[Definition~2.3]{das-kim2024}, depends in general on the choice of the partition sequence. For $p\ge1$, consider the Banach space introduced in \cite[Definition~2.4]{das-kim2024},
\begin{equation}\label{eq:Xp}
    \mathcal X_\pi^p:=\Big\{x\in C^0([0,T]):\sup_{n\ge0}[x]_{\pi^n}^{(p)}(T)<\infty\Big\},
\end{equation}
equipped with the norm
\begin{equation}\label{eq:Xp-norm}
    \|x\|_\pi^{(p)}:=\max\bigg\{\|x\|_\infty,\ |x(0)|+\sup_{n\ge0}\big([x]_{\pi^n}^{(p)}(T)\big)^{\frac1p}\bigg\}.
\end{equation}
Due to the continuity of $x \in C^0([0, T])$, the definition of the variation index implies
\begin{equation}\label{eq:index-threshold}
    \limsup_{n\to\infty}[x]_{\pi^n}^{(q)}(T)
    =\begin{cases}
       0, & q>p_\pi(x),\\
       \infty, & q<p_\pi(x).
    \end{cases}
\end{equation}
Thus, $p_\pi(x)$ separates the exponents $q>p_\pi(x)$, for which the variation sums tend to zero, from the exponents $1\le q<p_\pi(x)$, for which these sums are unbounded. No conclusion is imposed at $q=p_\pi(x)$.

In \cite[Theorem~4.3]{das-kim2024}, for $p>1$, the uniform boundedness $\sup_{n\ge0} [x]_{\pi^n}^{(p)}(T) < \infty$, equivalently, the membership $x\in \mathcal{X}_\pi^p$, was characterized in terms of so-called Schauder coefficients associated with a certain class of refining partition sequences. We shall introduce a very special type of Schauder coefficients (along the dyadic partition sequence $\T$) in Section~\ref{subsec:counterexamples}, and study relations among several function spaces in Section~\ref{sec:classical-pvar-spaces}. Here, however, we focus on seeking conditions under which the variation index in \eqref{eq:intro-index} is independent of the choice of partition sequence.

For any finite partition $\mathfrak P = \{0 = t_0 < t_1 < \cdots < t_n=T\}$ of $[0, T]$, we call $[u, v]$ an \emph{interval} of $\mathfrak P$ and write $[u,v]\in \mathfrak P$, if $u < v$ are consecutive partition points of $\mathfrak P$, i.e., $u=t_i$, $v=t_{i+1}$ for some $i = 0, \cdots, n-1$. For such partition $\mathfrak P$ of $[0,T]$, let
\[
    [x]_{\mathfrak P}^{(p)}(T) := \sum_{[u,v]\in \mathfrak P} \big|x(v)-x(u)\big|^p = \sum_{i=1}^n \big|x(t_i)-x(t_{i-1})\big|^p,
\]
where the first sum is understood to be taken over all consecutive points of $\mathfrak P$. We denote the classical $p$-variation seminorm by
\begin{equation}\label{eq:classical-pvar}
    \|x\|_{p\text{-}\cvar} := \|x\|_{p\text{-}\cvar;[0,T]} = \Big(\sup_{\mathfrak P} [x]_{\mathfrak P}^{(p)}(T)\Big)^{\frac1p},
\end{equation}
where the supremum is taken over all finite partitions of $[0,T]$, and we define
\begin{equation*}
    p_{\cvar}(x) := \inf \Big\{ p\ge 1: \|x\|_{p\text{-}\cvar} < \infty \Big\}.
\end{equation*}
Hence, $p_{\cvar}(x)$ is the critical threshold for the finiteness of the classical $q$-variation. Finiteness at the critical exponent $q=p_{\cvar}(x)$ is not asserted.

Since every $\pi^n$ of a partition sequence $\pi$ is one of the partitions appearing in \eqref{eq:classical-pvar}, one always has
\begin{equation}\label{eq:easy-direction}
    p_\pi(x)\le p_{\cvar}(x).
\end{equation}
The main point is to identify assumptions under which the opposite inequality holds.

\medskip

\subsection{Invariance of variation index} \label{subsec:invariance}

The following is the main result of this section.

\begin{theorem}[Variation index invariance] \label{thm:main-intro}
    Let $\pi=(\pi^n)_{n\ge 0}$ be a balanced refining sequence of partitions of $[0,T]$ with vanishing mesh. Suppose that
    \begin{equation}\label{eq:no-large-jump-intro}
        B_\pi := \sup_{n\ge 0}\frac{|\pi^n|}{|\pi^{n+1}|}<\infty.
    \end{equation}
    Then, for every $\alpha\in(0,1]$ and every $x\in C^\alpha([0,T])$,
    \begin{equation}\label{eq:main-equality-intro}
        p_\pi(x)=p_{\cvar}(x)=p_\T(x).
    \end{equation}
\end{theorem}

Before proving Theorem~\ref{thm:main-intro} in the next subsection, we make a few remarks on this result.

\begin{remark} [No lower bound on the mesh ratio]
    From Definition~3.4 of \cite{das2021}, we say that a refining partition sequence $\pi = (\pi^n)_{n \ge 0}$ is \emph{complete refining} if there exist positive constants $a, b$ such that
    \begin{equation}    \label{con:complete-refining}
        1+a \le \frac{|\pi^n|}{|\pi^{n+1}|} \le b \quad \text{holds for all } n \ge 0.
    \end{equation}
    We note that condition \eqref{eq:no-large-jump-intro} imposes no lower bound strictly larger than one on the ratio $|\pi^n|/|\pi^{n+1}|$. Hence, Theorem~\ref{thm:main-intro} applies, in particular, to every balanced complete refining partition sequence, but no lower bound strictly larger than one is required.
\end{remark}

\smallskip

\begin{remark} [Membership at the critical exponent]    \label{rem:Banach-space-membership}
    The equality in \eqref{eq:main-equality-intro} concerns only the location of the critical exponent and should not be confused with equality of the corresponding spaces at the critical exponent. Indeed, suppose that $\pi$ and $\sigma$ both satisfy the assumptions of Theorem~\ref{thm:main-intro}, and let $x\in C^\alpha([0,T])$ for some $\alpha>0$. Suppose that the common value $p:=p_\pi(x)=p_\sigma(x)=p_{\cvar}(x)$ is finite. Then \eqref{eq:index-threshold} implies
    \[
        x\in\mathcal X_\pi^q\cap\mathcal X_\sigma^q \quad\text{for every }q>p, \qquad x\notin\mathcal X_\pi^q\cup\mathcal X_\sigma^q \quad\text{for every }1\le q<p.
    \]
    Thus, for this fixed path $x$, membership in $\mathcal X_\pi^q$ and $\mathcal X_\sigma^q$ agrees at every noncritical exponent $q\ne p$. At the critical exponent $q=p$, however, Theorem~\ref{thm:main-intro} gives no information on the membership; it may depend on the partition sequence. In Proposition~\ref{prop:critical-space-noninvariance} of Subsection~\ref{subsec:critical-space-noninvariance}, we construct two balanced complete refining partition sequences $\pi$ and $\sigma$ and a H\"older continuous function $x$ such that
    \[
        p_\pi(x)=p_\sigma(x)=p, \qquad x\in\mathcal X_\sigma^p\setminus\mathcal X_\pi^p.
    \]
\end{remark}

\smallskip

\begin{remark}[Finite-dimensional paths]\label{rem:vector-valued-extension}
    Theorem~\ref{thm:main-intro} extends immediately to continuous paths with values in $\R^d$. For $x=(x^1,\ldots,x^d)\in C^0([0,T];\R^d)$, define the discrete and classical variations using any norm on $\R^d$. By equivalence of norms in finite dimensions,
    \[
        p_\pi(x)=\max_{1\le i\le d}p_\pi(x^i), \qquad p_{\cvar}(x)=\max_{1\le i\le d}p_{\cvar}(x^i).
    \]
    Consequently, under the assumptions of Theorem~\ref{thm:main-intro},
    \[
        p_\pi(x)=p_{\cvar}(x)=p_\T(x), \qquad x\in C^\alpha([0,T];\R^d).
    \]
\end{remark}

\medskip

\subsection{Regular skeletons and proof of Theorem~\ref{thm:main-intro}} \label{subsec:proofs}

We first introduce some useful concepts and establish preliminary results. The proof of Theorem~\ref{thm:main-intro} is given at the end of this subsection.

Let $\rho=(\rho^m)_{m\ge 0}$ be a refining subsequence of $\pi \in \Pi([0,T])$, so that
\begin{equation}\label{eq:subsequence}
    \rho^m=\pi^{n_m}, \qquad 0\le n_0<n_1<\cdots.
\end{equation}
For every interval $I$ of $\rho^m$, i.e., two consecutive points of $\rho^m$, we denote by
\[
    \mathcal{J}_{m}(I) :=\{J:J\text{ is an interval of }\rho^{m+1},\ J\subset I\}
\]
the collection of its children, and set
\begin{equation}\label{eq:branching-number}
    M_m(\rho) := \max_{I\in\rho^m} \#\big(\mathcal{J}_m(I)\big).
\end{equation}
Here, $\# A$ represents the cardinality of a set $A$.

\begin{definition}[Finitely refining skeleton]
    A refining subsequence $\rho =(\pi^{n_m})_{m\ge0}$ of $\pi$ is called a \emph{finitely refining skeleton} of $\pi$, if
    \[
        \sup_{m\ge0}M_m(\rho)<\infty.
    \]
\end{definition}

\begin{definition}[Regular skeleton]    \label{def:regular-skeleton}
    A finitely refining skeleton $\rho=(\pi^{n_m})_{m\ge0}$ of $\pi$ is called \emph{regular}, if
    \begin{equation}    \label{con:regular-skeleton}
        \sum_{m=0}^\infty |\rho^m|^\gamma<\infty, \qquad \text{for every }\gamma>0.
    \end{equation}
\end{definition}

\begin{remark}\label{rem:regular-skeleton}
    If there exists a constant $\lambda\in(0,1)$ such that $|\rho^{m+1}|\le\lambda|\rho^m|$ holds for every $m\ge 0$, then condition \eqref{con:regular-skeleton} is satisfied.
\end{remark}

For $x\in C^0([0,T])$, let $x_m^\rho$ denote the piecewise linear approximation of $x$ along $\rho^m$, namely
\begin{equation}\label{eq:linear-interpolation}
    x_m^\rho(t) := x(u)+\frac{x(v)-x(u)}{v-u}(t-u),
    \qquad t\in[u,v], \quad [u,v]\in\rho^m.
\end{equation}
We define the $m$-th interpolation block by
\begin{equation}\label{eq:block-definition}
    u_m^\rho:=x_{m+1}^\rho-x_m^\rho.
\end{equation}
The function $u_m^\rho$ vanishes at every point of $\rho^m$ and is piecewise linear along $\rho^{m+1}$.

The modulus of continuity of $x$ is denoted by
\begin{equation}\label{eq:modulus}
    \omega_x(h) := \sup \Big\{|x(t)-x(s)|:s,t\in[0,T],\ |t-s|\le h \Big\}.
\end{equation}
The following elementary bound will be used repeatedly:
\begin{equation}\label{eq:block-modulus}
    \|u_m^\rho\|_\infty \le \omega_x(|\rho^m|).
\end{equation}
Indeed, on each interval of $\rho^m$, both $x_m^\rho$ and $x_{m+1}^\rho$ are convex combinations of values of $x$ taken inside the same interval.

\begin{lemma}[Lower semicontinuity of the $p$-variation]\label{lem:pvar-lower-semicontinuity}
    Let $p\ge1$ and let $(f_n)_{n\ge1}\subset C^0([0,T])$ converge uniformly to $f\in C^0([0,T])$. Then
    \begin{equation}\label{eq:pvar-lower-semicontinuity}
        \|f\|_{p\text{-}\cvar}
        \le \liminf_{n\to\infty}\|f_n\|_{p\text{-}\cvar}.
    \end{equation}
\end{lemma}

\begin{proof}
    Fix a finite partition $\mathfrak P$ of $[0,T]$. Since $f_n\to f$ uniformly,
    \[
        [f]_{\mathfrak P}^{(p)}(T)
        =\lim_{n\to\infty}[f_n]_{\mathfrak P}^{(p)}(T)
        \le\liminf_{n\to\infty}\|f_n\|_{p\text{-}\cvar}^p.
    \]
    Taking the supremum over all finite partitions $\mathfrak P$ of $[0,T]$ and then the $p$-th root proves \eqref{eq:pvar-lower-semicontinuity}.
\end{proof}

\begin{lemma}   \label{lem:interpolation-inequality}
    Let $1\le q<r<\infty$. Then, for every $f\in C^0([0,T])$,
    \begin{equation*}
        \|f\|_{r\text{-}\cvar} \le (2\|f\|_\infty)^{1-\frac qr} \big(\|f\|_{q\text{-}\cvar}\big)^{\frac qr}.
    \end{equation*}
\end{lemma}

\begin{proof}
    For every finite partition $\mathfrak P$, $[f]_{\mathfrak P}^{(r)}(T) \le (2\|f\|_\infty)^{r-q}[f]_{\mathfrak P}^{(q)}(T)$. Taking the supremum over $\mathfrak P$ and then the $r$-th root gives the result.
\end{proof}

We now establish an estimate for a single interpolation block.

\begin{lemma}   \label{lem:block-qvar}
    Let $q>1$, let $\rho^m\subset\rho^{m+1}$ be two finite partitions, and let $y$ be a continuous function which is piecewise linear along $\rho^{m+1}$ and vanishes on $\rho^m$. Then
    \begin{equation}\label{eq:block-qvar-general}
        \|y\|_{q\text{-}\cvar}
        \le 2^{1-\frac1q} M_m(\rho)^{1-\frac1q}
        \Big(\sum_{[u, v]\in\rho^{m+1}}|y(v)-y(u)|^q\Big)^{\frac1q}.
    \end{equation}
\end{lemma}

\begin{proof}
    Fix an interval $I$ of $\rho^m$. Since $y$ is piecewise linear along $\rho^{m+1}$,
    \[
        \|y\|_{q\text{-}\cvar;I} \le \operatorname{TV}(y;I) = \sum_{[u, v] \in \mathcal{J}_m(I)}|y(v)-y(u)|.
    \]
    By H\"older's inequality,
    \[
        \|y\|_{q\text{-}\cvar;I}^q \le M_m(\rho)^{q-1} \sum_{[u, v] \in \mathcal{J}_m(I)} |y(v)-y(u)|^q.
    \]
    When a test partition for the classical $q$-variation crosses a point of $\rho^m$, we insert that point and use the fact that $y$ vanishes there. The inequality $|a-b|^q\le 2^{q-1}(|a|^q+|b|^q)$ then yields
    \[
        \|y\|_{q\text{-}\cvar}^q \le 2^{q-1}\sum_{I\in\rho^m} \|y\|_{q\text{-}\cvar;I}^q.
    \]
    Combining the last two inequalities proves \eqref{eq:block-qvar-general}.
\end{proof}

\begin{proposition}[Subcritical interpolation]  \label{prop:subcritical}
    Let $\pi\in\Pi([0,T])$, let $\rho=(\pi^{n_m})_{m\ge 0}$ be a refining subsequence of $\pi$, and let $1<q<r<\infty$. Suppose that $x\in \mathcal{X}_\pi^q$. Then
    \begin{align}\label{eq:subcritical-bound}
        \|x\|_{r\text{-}\cvar}
        \le \|x_0^\rho\|_{r\text{-}\cvar} + 2^{1+\frac{q-1}{r}} \Big(\sup_{n\ge 0}[x]_{\pi^n}^{(q)}(T)\Big)^{\frac1r} \sum_{m=0}^{\infty} M_m(\rho)^{\frac{q-1}{r}} \|u_m^\rho\|_\infty^{1-\frac qr},
    \end{align}
    provided that the series on the right-hand side is finite. Moreover, for every $N\ge0$,
    \begin{align}\label{eq:subcritical-tail}
        \|x-x_N^\rho\|_{r\text{-}\cvar} \le 2^{1+\frac{q-1}{r}} \Big(\sup_{n\ge0}[x]_{\pi^n}^{(q)}(T)\Big)^{\frac1r} \sum_{m=N}^{\infty}M_m(\rho)^{\frac{q-1}{r}} \|u_m^\rho\|_\infty^{1-\frac qr},
    \end{align}
    provided that the series on the right-hand side is finite.
\end{proposition}

\begin{proof}
    Set
    \[
        |x|^{(q)}_{\pi} := \sup_{n\ge0} \big([x]_{\pi^n}^{(q)}(T)\big)^{\frac1q} < \infty.
    \]
    Fix $m\ge 0$, an interval $I=[a,b]$ of $\rho^m$, and a child interval $J=[s,t]\subset I$ of $\rho^{m+1}$. By \eqref{eq:linear-interpolation},
    \begin{equation}\label{eq:block-increment}
        u_m^\rho(t)-u_m^\rho(s) = x(t)-x(s)-\frac{t-s}{b-a}\big(x(b)-x(a)\big).
    \end{equation}
    Consequently,
    \begin{align}
        \sum_{[s,t]\in\rho^{m+1}} |u_m^\rho(t)-u_m^\rho(s)|^q
        &\le 2^{q-1}\bigg( [x]_{\rho^{m+1}}^{(q)}(T) + \sum_{[a,b]\in\rho^m}|x(b)-x(a)|^q \sum_{J\in\mathcal{J}_m([a,b])} \Big(\frac{|J|}{b-a}\Big)^q \bigg)  \nonumber
        \\
        &\le 2^q\big(|x|^{(q)}_{\pi}\big)^q. \label{eq:nodal-qenergy}
    \end{align}
    Here, the last inequality used
    \[
        \sum_{J\in\mathcal{J}_m(I)} \Big(\frac{|J|}{|I|}\Big)^q
        \le\sum_{J\in\mathcal{J}_m(I)}\frac{|J|}{|I|}=1.
    \]
    Lemma~\ref{lem:block-qvar} applied to $u_m^\rho$ and \eqref{eq:nodal-qenergy} imply
    \begin{equation}\label{eq:um-qvar}
        \|u_m^\rho\|_{q\text{-}\cvar} \le 2^{2-\frac1q} M_m(\rho)^{1-\frac1q} |x|^{(q)}_{\pi}.
    \end{equation}
    
    Applying Lemma~\ref{lem:interpolation-inequality} to $u_m^\rho$, followed by \eqref{eq:um-qvar}, gives
    \begin{equation}\label{eq:um-rvar}
        \|u_m^\rho\|_{r\text{-}\cvar} \le 2^{1+\frac{q-1}{r}} M_m(\rho)^{\frac{q-1}{r}} \big(|x|^{(q)}_{\pi}\big)^{\frac qr} \|u_m^\rho\|_\infty^{1-\frac qr}.
    \end{equation}
    For integers $K>N$, we have the telescoping representation
    \[
        x_K^\rho-x_N^\rho = \sum_{m=N}^{K-1}u_m^\rho.
    \]
    Hence, the triangle inequality for the classical $r$-variation seminorm and \eqref{eq:um-rvar} yield
    \[
        \|x_K^\rho-x_N^\rho\|_{r\text{-}\cvar}
        \le 2^{1+\frac{q-1}{r}}
        \Big(\sup_{n\ge0}[x]_{\pi^n}^{(q)}(T)\Big)^{\frac1r}
        \sum_{m=N}^{K-1}M_m(\rho)^{\frac{q-1}{r}}
        \|u_m^\rho\|_\infty^{1-\frac qr}.
    \]
    Since $x_K^\rho\to x$ uniformly as $K\to\infty$, the lower semicontinuity of $r$-variation from Lemma~\ref{lem:pvar-lower-semicontinuity} gives \eqref{eq:subcritical-tail}. Taking $N=0$ and using $x=x_0^\rho+(x-x_0^\rho)$ then gives \eqref{eq:subcritical-bound}.
\end{proof}

\begin{theorem}[Finitely refining criterion] \label{thm:bounded-branching}
    Let $\pi\in\Pi([0,T])$ contain a finitely refining skeleton $\rho=(\pi^{n_m})_{m\ge0}$. If $x\in C^0([0,T])$ satisfies
    \begin{equation}\label{eq:all-powers-summability}
        \sum_{m=0}^{\infty} \big(\|u_m^\rho\|_\infty\big)^\delta < \infty \qquad \text{for every } \delta > 0,
    \end{equation}
    then $p_\pi(x)=p_{\cvar}(x)$.
\end{theorem}

\begin{proof}
    We use the convention $\inf \emptyset =\infty$. Due to \eqref{eq:easy-direction}, it is enough to show $p_\pi(x) \ge p_{\cvar}(x)$. If $p_\pi(x)=\infty$, there is nothing to prove. Suppose that $p_\pi(x)<\infty$. Fix $r>p_\pi(x)$ and choose $q$ such that $p_\pi(x)<q<r$. Then $x\in \mathcal{X}_\pi^q$ by \eqref{eq:index-threshold}. Since $\rho$ is a finitely refining skeleton, we may set
    \[
        M_\rho:=\sup_{m\ge0}M_m(\rho)<\infty.
    \]
    Applying \eqref{eq:all-powers-summability} with
    $\delta=1-q/r>0$, Proposition~\ref{prop:subcritical} yields $\|x\|_{r\text{-}\cvar}<\infty$, thus $p_{\cvar}(x)\le r$. Sending $r\downarrow p_\pi(x)$ proves the desired inequality.
\end{proof}

\begin{corollary}[Positive H\"older regularity]\label{cor:holder-skeleton}
    Let $\pi\in\Pi([0,T])$ contain a regular skeleton. Then, for every $\alpha\in(0,1]$ and $x\in C^{\alpha}([0,T])$, we have the identity $p_\pi(x)=p_{\cvar}(x)$.
\end{corollary}

\begin{proof}
    Let $\rho=(\pi^{n_m})_{m\ge0}$ be a regular skeleton of $\pi$. Fix $\delta>0$. Since $x\in C^{\alpha}([0,T])$, the inequality \eqref{eq:block-modulus} gives
    \[
        \|u_m^\rho\|_\infty \le \omega_x(|\rho^m|) \le |x|_{C^{\alpha}}|\rho^m|^\alpha.
    \]
    Consequently,
    \[
        \sum_{m=0}^\infty \big(\|u_m^\rho\|_\infty\big)^\delta
        \le |x|_{C^{\alpha}}^\delta \sum_{m=0}^\infty |\rho^m|^{\alpha\delta} <\infty,
    \]
    where the last inequality follows from
    \eqref{con:regular-skeleton} with
    $\gamma=\alpha\delta>0$. Hence the condition \eqref{eq:all-powers-summability} holds and
    Theorem~\ref{thm:bounded-branching} applies.
\end{proof}

We now show that the hypotheses of Theorem~\ref{thm:main-intro} produce a regular skeleton.

\begin{lemma}[Construction of a regular skeleton]\label{lem:skeleton-construction}
    Let $\pi\in\Pi([0,T])$ be balanced with constant $c$, and suppose that $B_\pi$ in \eqref{eq:no-large-jump-intro} is finite. Then $\pi$ contains a regular skeleton.
\end{lemma}

\begin{proof}
    Starting from $n_0=0$, define inductively
    \begin{equation}\label{eq:skeleton-indices}
        n_{m+1} :=\min\Big\{n>n_m: |\pi^n| \le \frac12 |\pi^{n_m}|\Big\}.
    \end{equation}
    The index $n_{m+1}$ exists because $|\pi^n|\to 0$. By the minimality of $n_{m+1}$,
    \[
        |\pi^{n_{m+1}-1}|>\frac12|\pi^{n_m}|.
    \]
    Indeed, this is immediate if $n_{m+1}=n_m+1$, while otherwise it follows because $n_m<n_{m+1}-1<n_{m+1}$. On the other hand, by the definition of $B_\pi$,
    \[
        |\pi^{n_{m+1}}| \ge \frac1{B_\pi}|\pi^{n_{m+1}-1}| > \frac1{2B_\pi}|\pi^{n_m}|.
    \]
    Consequently,
    \begin{equation}\label{eq:skeleton-mesh-comparison}
        \frac{1}{2B_\pi}|\pi^{n_m}| < |\pi^{n_{m+1}}| \le \frac12 |\pi^{n_m}|.
    \end{equation}
    Set $\rho^m:=\pi^{n_m}$. From \eqref{eq:skeleton-mesh-comparison},
    \[
        |\rho^{m+1}|\le\frac12|\rho^m|, \qquad \text{hence} \qquad  |\rho^m|\le 2^{-m}|\rho^0|.
    \]
    Therefore, for every $\gamma>0$,
    \[
        \sum_{m=0}^\infty|\rho^m|^\gamma \le |\rho^0|^\gamma \sum_{m=0}^\infty2^{-m\gamma} <\infty.
    \]
    
    Let $I$ be an interval of $\rho^m$. Every interval $J$ of $\rho^{m+1}$ contained in $I$ has length at least
    \[
        \underline{\rho^{m+1}} \ge \frac{|\rho^{m+1}|}{c} > \frac{|\rho^m|}{2cB_\pi}.
    \]
    Since $|I|\le|\rho^m|$, the number of such intervals of $J$ in $I$ is bounded by $2cB_\pi$. This proves that $\sup_{m\ge 0}M_m(\rho)\le\lceil 2cB_\pi\rceil$, and $\rho$ is a regular skeleton.
\end{proof}

The proof of Theorem~\ref{thm:main-intro} is now immediate.

\begin{proof}[Proof of Theorem~\ref{thm:main-intro}]
    Lemma~\ref{lem:skeleton-construction} and Corollary~\ref{cor:holder-skeleton} give $p_\pi(x)=p_{\cvar}(x)$. The dyadic partition sequence $\T$ is itself a regular skeleton: each interval of $\T^m$ has exactly two children, and $|\T^m|=T2^{-m}$. Thus, for every $\gamma>0$,
    \[
        \sum_{m=0}^\infty|\T^m|^\gamma = T^\gamma\sum_{m=0}^\infty2^{-m\gamma} <\infty.
    \]
    Corollary~\ref{cor:holder-skeleton} therefore gives $p_\T(x)=p_{\cvar}(x)$.
\end{proof}

\medskip

\subsection{Beyond H\"older regularity} \label{subsec:beyond-holder}

Positive H\"older regularity is a convenient sufficient condition in
Corollary~\ref{cor:holder-skeleton}, but it is not the intrinsic assumption in Theorem~\ref{thm:bounded-branching}; the relevant condition is the summability of the interpolation errors along a finitely refining skeleton. Thus, the purpose of this subsection is to establish the result of Theorem~\ref{thm:main-intro} under a weaker regularity condition on $x$.

We first recall the notation \eqref{eq:modulus} for the modulus of continuity and provide the following definition.

\begin{definition}[All-orders Dini modulus] \label{def:all-orders-dini}
    We say that $x\in C^0([0,T])$ has an all-orders Dini modulus if
    \begin{equation}\label{eq:all-orders-dini}
        \int_0^T\omega_x(h)^\delta\,\frac{dh}{h}<\infty, \qquad \text{for every }\delta>0.
    \end{equation}
\end{definition}

If $x\in C^\alpha([0,T])$ for some $\alpha \in (0, 1]$, then for every $\delta>0$,
\[
    \int_0^T\omega_x(h)^\delta\,\frac{dh}{h} \le |x|_{C^\alpha}^\delta \int_0^T h^{\alpha\delta-1}\,dh =\frac{T^{\alpha\delta}}{\alpha\delta}|x|_{C^\alpha}^\delta < \infty.
\]
Thus, every positively H\"older continuous path has an all-orders Dini modulus.

\begin{proposition}[Geometric skeletons]\label{prop:dini-geometric}
    Let $\pi\in\Pi([0,T])$ contain a finitely refining skeleton $\rho=(\pi^{n_m})_{m\ge0}$ for which there exists $\lambda\in(0,1)$ such that
    \begin{equation}\label{eq:geometric-skeleton-again}
        |\rho^{m+1}|\le\lambda|\rho^m|, \qquad m\ge0.
    \end{equation}
    If $x$ has an all-orders Dini modulus, then $p_\pi(x)=p_{\cvar}(x)$.
\end{proposition}

\begin{proof}
    Fix $\delta>0$. For $m\ge1$, the monotonicity of the modulus of continuity gives
    \[
        \omega_x(|\rho^m|)^\delta\log\frac1\lambda \le \int_{|\rho^m|}^{|\rho^m|/\lambda} \omega_x(h)^\delta\,\frac{dh}{h}.
    \]
    By \eqref{eq:geometric-skeleton-again}, the intervals $[|\rho^m|,|\rho^m|/\lambda]$, $m\ge1$, have pairwise disjoint interiors and are contained in $(0,T]$. Hence
    \[
        \sum_{m=1}^\infty\omega_x(|\rho^m|)^\delta \le \frac1{\log(1/\lambda)} \int_0^T\omega_x(h)^\delta\,\frac{dh}{h}<\infty.
    \]
    Since the initial term $\omega_x(|\rho^0|)^\delta$ is finite, the function $x\in C^0([0,T])$ satisfies
    \begin{equation}\label{eq:discrete-modulus-condition}
        \sum_{m=0}^\infty \omega_x(|\rho^m|)^\delta<\infty \qquad \text{for every }\delta>0.
    \end{equation}
    By \eqref{eq:block-modulus}, $\|u_m^\rho\|_\infty^\delta \le \omega_x(|\rho^m|)^\delta$. Thus \eqref{eq:discrete-modulus-condition} implies \eqref{eq:all-powers-summability}, and Theorem~\ref{thm:bounded-branching} applies.
\end{proof}

\begin{corollary}[All-orders Dini extension of Theorem~\ref{thm:main-intro}]    \label{cor:dini-main}
    Let $\pi$ satisfy the assumptions of Theorem~\ref{thm:main-intro}. If $x\in C^0([0,T])$ has an all-orders Dini modulus, then $p_\pi(x)=p_{\cvar}(x)=p_\T(x)$.
\end{corollary}

\begin{proof}
    The skeleton constructed in Lemma~\ref{lem:skeleton-construction} satisfies
    \[
        |\rho^{m+1}|\le\frac12|\rho^m|.
    \]
    Proposition~\ref{prop:dini-geometric} therefore yields $p_\pi(x)=p_{\cvar}(x)$. The dyadic partition sequence is itself a finitely refining geometric skeleton with ratio $1/2$, and the same proposition gives $p_\T(x)=p_{\cvar}(x)$.
\end{proof}

The following example gives a function which has an all-orders Dini modulus but fails to be H\"older continuous for every positive exponent. Consequently, Corollary~\ref{cor:dini-main} genuinely extends Theorem~\ref{thm:main-intro} beyond the union of the positive H\"older classes.

\begin{example} [The extension is strictly beyond H\"older regularity]\label{ex:non-holder-modulus}
    Let $\beta\in(0,1)$ and define $x_\beta\in C^0([0,T])$ by
    \begin{equation}\label{eq:stretched-log-modulus}
        x_\beta(0):=0,\qquad x_\beta(t):=\exp\bigg[-\Big(\log\frac{eT}{t}\Big)^\beta\bigg], \qquad 0<t\le T.
    \end{equation}
    For every $\delta>0$, the change of variables $u=\log(eT/h)$ gives
    \begin{equation}\label{eq:stretched-log-dini}
        \int_0^T x_\beta(h)^\delta\,\frac{dh}{h} = \int_1^\infty e^{-\delta u^\beta}\,du < \infty.
    \end{equation}
    Moreover, $x_\beta$ is increasing and concave on $[0,T]$. Indeed, writing
    \[
        L(t):=\log\frac{eT}{t}, \qquad 0<t\le T,
    \]
    a direct calculation gives
    \[
        x_\beta'(t) = \frac{\beta x_\beta(t)}{t}L(t)^{\beta-1}>0, \qquad x_\beta''(t) = \frac{\beta x_\beta(t)}{t^2}L(t)^{\beta-2} \big(\beta L(t)^\beta-L(t)+1-\beta\big).
    \]
    Since $L(t)\ge1$ and $0<\beta<1$, one has
    \[
        \beta L(t)^\beta-L(t)+1-\beta \le \beta L(t)-L(t)+1-\beta = (1-\beta)(1-L(t)) \le 0,
    \]
    hence $x_\beta''(t)\le0$ for $0<t\le T$.

    Since every increasing concave function vanishing at zero is subadditive, it follows that
    \[
        \omega_{x_\beta}(h)=x_\beta(h), \qquad 0\le h\le T.
    \]
    Hence, $x_\beta$ satisfies \eqref{eq:all-orders-dini}. On the other hand, for every $\alpha>0$,
    \begin{equation*}
        \frac{x_\beta(h)}{h^\alpha}\longrightarrow\infty \qquad\text{as } h\downarrow0.
    \end{equation*}
    Thus $x_\beta$ does not belong to $C^\alpha([0,T])$ for any $\alpha>0$.
\end{example}

\begin{remark}
    For a general regular skeleton in the sense of Definition~\ref{def:regular-skeleton}, the all-orders Dini condition need not imply \eqref{eq:discrete-modulus-condition}; the skeleton may sample many nearby logarithmic scales. The discrete condition \eqref{eq:discrete-modulus-condition} is therefore the appropriate general statement, while Proposition~\ref{prop:dini-geometric} uses the additional geometric separation of scales.
\end{remark}

\medskip

\subsection{Sharpness of the partition assumptions}    \label{subsec:counterexamples}

We show in this subsection that both assumptions on the partition sequence in Theorem~\ref{thm:main-intro}, namely the balancedness condition and condition \eqref{eq:no-large-jump-intro}, are meaningful. First, a sparse subsequence of dyadic partitions is perfectly balanced, but may fail \eqref{eq:no-large-jump-intro}; along such a sequence, the variation index can be equal to one even when the dyadic variation index is any prescribed number $s>1$. Second, if balancedness is removed, one may retain the exact mesh relation $|\pi^n|=2^{-n}$ and nevertheless force the variation index to be one by inserting sufficiently many small intervals adapted to the function.

For simplicity, we shall work on $[0,1]$. Let $\T^m=\{k2^{-m}:0\le k\le 2^m\}$ denote the dyadic partition at level $m$. In what follows, we introduce the Faber--Schauder representation of continuous functions on $[0,1]$ along $\T = (\T^m)_{m \ge 0}$.

We first consider the Haar basis for each $m \ge 0$ and $k \in I_m := \{0, 1, \cdots, 2^m -1\}$
\begin{equation}\label{Eq:Haar}
    \psi^{\T}_{m, k}(t) := 2^{\frac{m}{2}}\psi(2^m t-k), \quad t \in [0, 1], \qquad \text{where} \qquad
    \psi(t) :=
    \begin{cases}
        ~1, &\qquad \text{if } t \in [0, \frac{1}{2}),
        \\
        -1, &\qquad \text{if } t \in [\frac{1}{2}, 1),
        \\
        ~0, &\qquad \text{otherwise}.
    \end{cases}
\end{equation}
Note that $\{ \mathbf{1}_{[0,1)} \} \cup \{\psi^{\T}_{m, k}\}_{m \ge 0, k \in I_m}$ is an orthonormal basis of $L^2([0, 1])$ with respect to the inner product $\langle f, g \rangle = \int_0^1 f(t)g(t)dt$. The Faber--Schauder functions $\{e^{\T}_{m, k}\}_{m \ge 0, k \in I_m}$ are defined by integrating $\psi^{\T}_{m, k}$
\begin{equation}    \label{Eq: e_mk}
    e^{\T}_{m, k}(t) := \int_0^t \psi^{\T}_{m, k}(s) \, ds, \qquad t \in [0, 1].
\end{equation}

Together with the affine functions $1$ and $t$, the Faber--Schauder functions $\{e_{m,k}^\T\}_{m\ge0,\,k\in I_m}$ form a Schauder basis of $C^0([0,1])$. Consequently, every $x\in C^0([0,1])$ admits the unique representation
\begin{equation}    \label{Eq: Schauder representation}
    x(t) = x(0) + \big(x(1)-x(0)\big)t+ \sum_{m=0}^{\infty}\sum_{k \in I_m} \theta^{x, \T}_{m, k} e^{\T}_{m, k}(t), \qquad t \in [0, 1],
\end{equation}
where the real numbers $\{\theta^{x,\T}_{m, k}\}_{m \ge 0, k \in I_m}$ are known as the \textit{Faber--Schauder coefficients of $x$}. Each coefficient admits a closed-form expression in terms of the function values of $x$ at the dyadic points:
\begin{equation}\label{eq:theta_mk}
    \theta^{x,\T}_{m, k} = 2^{\frac{m}{2}}\bigg(2 x\Big(\frac{2k+1}{2^{m+1}}\Big)-x\Big(\frac{k}{2^m}\Big) -  x\Big(\frac{k+1}{2^{m}}\Big) \bigg), \qquad m \ge 0, ~ k \in I_m.
\end{equation}

We also introduce a convenient family of dyadic tent functions. For $m\ge 0$ and $k \in I_m$, let $\varphi_{m,k}$ be the continuous function supported on $[k2^{-m},(k+1)2^{-m}]$ which is linear on each half of its support, vanishes at the endpoints, and takes the value one at the midpoint. Define
\begin{equation}\label{eq:triangular-wave}
    \Phi_m(t):=\sum_{k=0}^{2^m-1}\varphi_{m,k}(t), \qquad m \ge 0,
\end{equation}
hence $\Phi_m$ is the height-one dyadic triangular wave at scale $2^{-m}$. Since
\begin{equation}\label{eq:normalized-tent}
    \varphi_{m,k}=2^{\frac{m}{2}+1}e_{m,k}^{\T}, \qquad m \ge 0, \quad k \in I_m,
\end{equation}
a function of the form
\[
    x(t)=\sum_{m=0}^{\infty}a_m\Phi_m(t), \qquad t \in [0, 1]
\]
has Faber--Schauder coefficients
\begin{equation}\label{eq:wave-coefficients}
    \theta_{m,k}^{x,\T}=2^{\frac{m}{2}+1}a_m, \qquad k \in I_m.
\end{equation}
The dyadic case of \cite[Corollary~4.4]{das-kim2024} gives the characterization of the variation index of \eqref{eq:intro-index}, in terms of the Faber--Schauder coefficients:
\begin{equation}\label{eq:dyadic-index-coeff}
    p_\T(x) = \inf\Big\{p>1: \limsup_{m\to\infty} \xi^{(p)}_m(x) < \infty\Big\}, \qquad \xi^{(p)}_m(x) := 2^{-\frac{mp}{2}}\sum_{k=0}^{2^m-1} |\theta_{m,k}^{x,\T}|^p.
\end{equation}
Substituting \eqref{eq:wave-coefficients} into \eqref{eq:dyadic-index-coeff} gives
\begin{equation}\label{eq:xi-wave}
    \xi^{(p)}_m(x) = 2^p2^m|a_m|^p.
\end{equation}

We shall also use the standard estimate
\begin{equation}\label{eq:wave-holder}
    \sum_{j\ge0}a_j\Phi_j\in C^\alpha([0,1]) \qquad \text{if} \quad |a_j| \le C2^{-\alpha j}, \quad \alpha\in(0,1).
\end{equation}
For completeness, one may verify \eqref{eq:wave-holder} by using $\|\Phi_j\|_\infty\le 1$ and $\operatorname{Lip}(\Phi_j)\le 2^{j+1}$, splitting the series at the unique scale comparable to $|t-s|$.

The first example shows that condition \eqref{eq:no-large-jump-intro} cannot be removed, even if the partition sequence is perfectly balanced.

\begin{proposition}[A sparse dyadic counterexample]\label{prop:sparse-dyadic}
    Fix $s>1$. There exist a function $x_s\in C^{1/s}([0,1])$ and a balanced refining partition sequence $\pi$ such that
    \begin{equation}\label{eq:sparse-counterexample-result}
        p_\T(x_s)=s > 1 = p_\pi(x_s).
    \end{equation}
    Moreover, $\pi$ may be chosen as a subsequence of the dyadic partition sequence.
\end{proposition}

\begin{proof}
    Choose a strictly increasing sequence of integers $(m_n)_{n\ge 1}$ satisfying
    \begin{equation}\label{eq:sparse-growth}
        \frac{m_{n+1}}{m_n}\longrightarrow\infty.
    \end{equation}
    Set $\pi^0:=\T^0$ and $\pi^n:=\T^{m_n}$ for $n\ge1$. The sequence $\pi$ is refining and perfectly balanced, since $|\pi^n| =\underline{\pi^n}$, whereas
    \[
        \frac{|\pi^n|}{|\pi^{n+1}|} = 2^{m_{n+1}-m_n}
    \]
    is unbounded. Define
    \begin{equation}\label{eq:sparse-function}
        x_s(t) :=\sum_{n=1}^{\infty}2^{-\frac{m_n+1}{s}}\Phi_{m_n+1}(t).
    \end{equation}
    The series converges uniformly, and \eqref{eq:wave-holder} gives $x_s\in C^{1/s}([0,1])$. By \eqref{eq:xi-wave}, the quantity $\xi^{(p)}_{m_n+1}(x_s)$ at an active level $m_n+1$ equals $\xi^{(p)}_{m_n+1}(x_s) = 2^p2^{(m_n+1)(1-\frac{p}{s})}$, and it is zero at all other levels. Equation \eqref{eq:dyadic-index-coeff} therefore yields $p_\T(x_s)=s$.

    It remains to compute the $p$-th variation along $\pi$. For $m>j$, a direct calculation gives
    \begin{equation}\label{eq:single-wave-variation}
        \Big(\big[2^{-\frac js}\Phi_j\big]_{\T^m}^{(p)}(1)\Big)^{\frac1p} = 2^{1-m(1-\frac1p)}2^{j(1-\frac1s)}.
    \end{equation}
    If $j\ge m$, then $\Phi_j$ vanishes at every point of $\T^m$, so its $p$-th variation along $\T^m$ is zero. Hence Minkowski's inequality and \eqref{eq:single-wave-variation} imply
    \begin{equation*}
        \Big([x_s]_{\pi^N}^{(p)}(1)\Big)^{\frac1p}
        = \Big([x_s]_{\T^{m_N}}^{(p)}(1)\Big)^{\frac1p} \le 2^{1-m_N(1-\frac1p)} \sum_{n<N} 2^{(m_n+1)(1-\frac1s)}
        \le C_s 2^{-m_N(1-\frac1p)+(m_{N-1}+1)(1-\frac1s)}
    \end{equation*}
    for some constant $C_s$ depending only on $s$. By \eqref{eq:sparse-growth}, the right-hand side tends to zero for every fixed $p>1$. Thus
    \[
        \lim_{N\to\infty}[x_s]_{\pi^N}^{(p)}(1)=0, \qquad p>1,
    \]
    which proves $p_\pi(x_s)=1$.
\end{proof}

Proposition~\ref{prop:sparse-dyadic} does not exploit irregular locations of partition points. Every level of $\pi$ is uniform; the failure is entirely caused by the increasingly large gaps between the sampled dyadic scales.

We next show that balancedness cannot be omitted from Theorem~\ref{thm:main-intro}, even when the largest mesh decreases exactly by a factor of two at every step.

\begin{lemma} \label{lem:small-energy-refinement}
    Let $f\in C^0([a,b])$, and let $P$ be a finite partition of $[a,b]$. For $q>1$ and $\varepsilon>0$, there exists a finite refinement $Q\supset P$ such that
    \begin{equation}\label{eq:small-energy-refinement}
        [f]_{Q}^{(q)}(b)<\varepsilon.
    \end{equation}
\end{lemma}

\begin{proof}
    It is enough to work on each interval $[u,v]$ of $P$. If $f(u)=f(v)$, no additional point is needed. Suppose that $f(v)>f(u)$; the opposite case is analogous. For $N\ge 1$, define the first hitting times
    \[
        t_j:=\inf\Big\{t\in[u,v]: f(t)=f(u)+\frac{j}{N}\big(f(v)-f(u)\big)\Big\}, \qquad 0\le j\le N.
    \]
    After adding $v$ if necessary, the resulting partition of $[u,v]$ has $q$-th variation $N^{1-q}|f(v)-f(u)|^q$. Choosing $N$ sufficiently large on each interval of $P$, with the error divided among the finitely many intervals, proves the result.
\end{proof}

\begin{proposition}[An unbalanced counterexample with regular largest mesh]\label{prop:unbalanced-counterexample}
    Fix $s>1$. There exist $y_s\in C^{1/s}([0,1])$ and a refining partition sequence $\pi$ such that
    \begin{equation}\label{eq:unbalanced-mesh}
        |\pi^n|=2^{-n}, \qquad n\ge 1,
    \end{equation}
    while
    \begin{equation}\label{eq:unbalanced-result}
        p_\T(y_s)=s > 1 = p_\pi(y_s).
    \end{equation}
    In particular, $|\pi^n|/|\pi^{n+1}|=2$ for every $n$, but $\pi$ is not balanced.
\end{proposition}

\begin{proof}
    Let
    \[
        z_s(t):=\sum_{j=0}^{\infty}2^{-\frac js}\Phi_j(t), \qquad t\in[0,1].
    \]
    By \eqref{eq:wave-holder}, $z_s\in C^{1/s}([0,1])$, and \eqref{eq:xi-wave}, \eqref{eq:dyadic-index-coeff} give $p_\T(z_s)=s$. Define
    \begin{equation}\label{eq:localized-function}
        y_s(t):=
        \begin{cases}
           c_s z_s(2t), \qquad & 0\le t\le\frac12,\\
           ~~0, \qquad & \frac12\le t\le 1,
        \end{cases}
    \end{equation}
    where $c_s>0$ is chosen so that
    \[
        \osc(y_s;[0,1]) := \max_{t\in[0,1]} y_s(t) - \min_{t\in[0,1]} y_s(t) \le 1.
    \]
    Since $z_s(0)=z_s(1)=0$, the function $y_s$ is $(1/s)$-H\"older continuous. The dyadic scaling in \eqref{eq:localized-function} also yields
    \begin{equation}\label{eq:localized-index}
        p_\T(y_s)=p_\T(z_s)=s.
    \end{equation}
    
    We now construct $\pi$ inductively. Set $\pi^0=\{0,1\}$. For $n\ge 1$, let
    \[
        q_n := 1+\frac1n, \qquad r_n := 2^{-n^2-2}.
    \]
    Assume that $\pi^{n-1}$ has been constructed and has no additional points in $[1/2,1]$ beyond those of $\T^{n-1}$. On $[0,1/2]$, apply Lemma~\ref{lem:small-energy-refinement} to the finite partition
    \[
        P_n := \big(\pi^{n-1}\cup\T^n\cup\{r_n\}\big)\cap[0,1/2]
    \]
    and obtain a refinement $Q_n\supset P_n$ satisfying
    \begin{equation}\label{eq:qn-small}
        [y_s]_{Q_n}^{(q_n)}(1/2)<2^{-n}.
    \end{equation}
    Set $\pi^n:=Q_n\cup\big(\T^n\cap[1/2,1]\big)$. Then $\pi^{n-1}\subset\pi^n$ and $\pi^n$ contains $\T^n$. Hence $|\pi^n|\le2^{-n}$. Since no point is added to the constant region $[1/2,1]$ beyond the dyadic points, there are intervals of length $2^{-n}$ in that region, and \eqref{eq:unbalanced-mesh} follows. On the other hand, $\pi^n$ contains both $0$ and $r_n$, so
    \[
        \underline{\pi^n}\le r_n=2^{-n^2-2}.
    \]
    Consequently,
    \[
        \frac{|\pi^n|}{\underline{\pi^n}} \ge 2^{n^2-n+2}\longrightarrow\infty,
    \]
    and $\pi$ is not balanced.
    
    Since $y_s$ vanishes on $[1/2,1]$, the inequality \eqref{eq:qn-small} gives $[y_s]_{\pi^n}^{(q_n)}(1) < 2^{-n}$. Fix $p>1$. For all sufficiently large $n$, one has $q_n<p$. Moreover, every increment of $y_s$ has absolute value at most one. Therefore,
    \[
        [y_s]_{\pi^n}^{(p)}(1) \le [y_s]_{\pi^n}^{(q_n)}(1) < 2^{-n}.
    \]
    It follows that $p_\pi(y_s)=1$. Together with \eqref{eq:localized-index}, this proves \eqref{eq:unbalanced-result}.
\end{proof}

\begin{remark}\label{rem:balance-vs-skeleton}
    Balancedness itself is not the intrinsic condition in Theorem~\ref{thm:main-intro}. An unbalanced partition sequence may still contain a regular skeleton, in which case Corollary~\ref{cor:holder-skeleton} remains applicable. Proposition~\ref{prop:unbalanced-counterexample} shows that, without either balancedness or a direct finite-branching assumption, the variation index can be reduced by inserting a rapidly increasing number of function-adapted partition points.
\end{remark}

\medskip

\subsection{Failure of critical-space invariance}\label{subsec:critical-space-noninvariance}

Theorem~\ref{thm:main-intro} identifies the variation index along a large class of partition sequences with the classical variation index. As pointed out in Remark~\ref{rem:Banach-space-membership}, this does not, however, imply equality of the corresponding spaces at the critical exponent. The following result shows that this failure at the critical exponent may occur even for two balanced complete refining partition sequences (recall \eqref{eq:balanced-intro} and \eqref{con:complete-refining}), one of which is a bi-Lipschitz deformation of the dyadic sequence.

\begin{proposition}[Failure of critical-space invariance]\label{prop:critical-space-noninvariance}
    For every $p>1$, there exist a function
    \[
        x\in\bigcap_{0<\alpha<1/p}C^\alpha([0,1])
    \]
    and two balanced complete refining partition sequences $\pi$ and $\sigma$ such that
    \begin{equation}\label{eq:critical-space-noninvariance}
        p_\pi(x)=p_\sigma(x)=p_{\cvar}(x)=p, \qquad x\in\mathcal X_\sigma^p\setminus\mathcal X_\pi^p.
    \end{equation}
    In particular, $\mathcal X_\pi^p\neq\mathcal X_\sigma^p$.
\end{proposition}

\begin{proof}
    We take $\sigma=\T$. For every integer $r\ge8$, $1\le s\le r$, and $1\le\ell\le2^{r^3-2r}$, define
    \begin{equation}\label{eq:endpoint-centers}
        z_{r,s,\ell}:=2^{-r}+4\Big((s-1)2^{r^3-2r}+\ell\Big)2^{-r^3}.
    \end{equation}
    The point $z_{r,s,\ell}$ is dyadic at level $r^3+s$; more precisely,
    \[
        k_{r,s,\ell}:=2^{r^3+s}z_{r,s,\ell} = 2^{r^3+s-r}+2^{s+2}\Big((s-1)2^{r^3-2r}+\ell\Big)
    \]
    is an integer. Let
    \[
        \varphi_{r,s,\ell}:=\varphi_{r^3+s,k_{r,s,\ell}}
    \]
    be the height-one dyadic tent supported on $[z_{r,s,\ell}, \, z_{r,s,\ell}+2^{-r^3-s}]$, and denote its peak by
    \begin{equation*}
        y_{r,s,\ell}:=z_{r,s,\ell}+2^{-r^3-s-1}.
    \end{equation*}

    The intervals
    \[
        D_r:=\big[2^{-r}, \, 3\cdot2^{-r-1}\big], \qquad r\ge8,
    \]
    are pairwise disjoint. Moreover, plugging $s = r$, $\ell = 2^{r^3-2r}$ into \eqref{eq:endpoint-centers} and adding $2^{-r^3}$ give
    \[
        2^{-r} + 4r\,2^{r^3-2r} 2^{-r^3}+2^{-r^3} = 2^{-r} + 4r\,2^{-2r}+2^{-r^3} < 2^{-r} + 2^{-r-1} = 3\cdot2^{-r-1}.
    \]
    Hence the intervals
    \begin{equation}\label{eq:local-deformation-intervals}
        L_{r,s,\ell}:=\big[z_{r,s,\ell}-2^{-r^3}, \, z_{r,s,\ell}+2^{-r^3}\big]
    \end{equation}
    are contained in $D_r$ and are pairwise disjoint. Since $2^{-r^3-s}\le\frac12\,2^{-r^3}$, the supports of all the tents $\varphi_{r,s,\ell}$ are also pairwise disjoint.

    Define
    \begin{equation}\label{eq:endpoint-noninvariant-function}
        x(t) := \sum_{r=8}^\infty 2^{-\frac{r^3-2r}{p}} \sum_{s=1}^{r}\sum_{\ell=1}^{2^{r^3-2r}} \varphi_{r,s,\ell}(t), \qquad t\in[0,1].
    \end{equation}
    Since the supports are pairwise disjoint and $2^{-\frac{r^3-2r}{p}} \xlongrightarrow{r\to\infty} 0$, the series in \eqref{eq:endpoint-noninvariant-function} converges uniformly to a continuous function.

    We first determine the H\"older regularity of $x$. By \eqref{eq:normalized-tent}, every nonzero Faber--Schauder coefficient at level $r^3+s$ is equal to
    \[
        \theta_{r^3+s,k}^{x,\T} = 2^{\frac{r^3+s}{2}+1}2^{-\frac{r^3-2r}{p}},
    \]
    and there are exactly $2^{r^3-2r}$ such coefficients at that level. We have
    \begin{equation*}
        2^{(r^3+s)(\alpha-\frac12)} |\theta_{r^3+s,k}^{x,\T}| = 2^{1-(\frac1p-\alpha)r^3+\frac{2r}{p}+\alpha s} \le2^{1-(\frac1p-\alpha)r^3+(\frac2p+\alpha)r}.
    \end{equation*}
    If $0<\alpha<1/p$, then the last expression is bounded uniformly in $r$, $s$, and $k$; the dyadic Ciesielski characterization \cite[Theorem~3.4]{fake_fBM} therefore gives
    \begin{equation}\label{eq:endpoint-function-holder}
        x\in C^\alpha([0,1]), \qquad 0<\alpha<\frac1p.
    \end{equation}

    At every active level $r^3+s$, the levelwise coefficient quantity of \eqref{eq:dyadic-index-coeff} satisfies, for $q>1$,
    \begin{equation}   \label{eq:endpoint-active-energy}
        \xi_{r^3+s}^{(q)}(x) =2^{-\frac{q(r^3+s)}2}2^{r^3-2r} \Big(2^{\frac{r^3+s}{2}+1} 2^{-\frac{r^3-2r}{p}}\Big)^q = 2^q\,2^{(r^3-2r)(1-\frac qp)}.
    \end{equation}
    At every other level, the corresponding quantity is zero. In particular, $\sup_{n\ge0}\xi_n^{(p)}(x)=2^p<\infty$, whereas $\xi_{r^3+s}^{(q)}(x) \xlongrightarrow{r\to\infty}\infty$ for every $q<p$. The dyadic coefficient characterization in \cite[Theorem~4.3 and Corollary~4.4]{das-kim2024} yields
    \begin{equation}\label{eq:endpoint-dyadic-membership}
        x\in\mathcal X_\T^p, \qquad p_\T(x)=p.
    \end{equation}

    We next construct a partition sequence which samples all the peaks in the $r$-th batch simultaneously. On each interval $L_{r,s,\ell}$ in \eqref{eq:local-deformation-intervals}, define
    \begin{equation}\label{eq:local-deformation}
        \Psi(t):=
        \begin{cases}
            \displaystyle
            t+2^{-s-1}\big(t-z_{r,s,\ell}+2^{-r^3}\big), & \qquad t \in [z_{r,s,\ell}-2^{-r^3}, \, z_{r,s,\ell}],\\[2mm]
            \displaystyle
            t+2^{-r^3-s-1}-2^{-s-1}\big(t-z_{r,s,\ell}\big), & \qquad t \in [z_{r,s,\ell}, \, z_{r,s,\ell}+2^{-r^3}].
        \end{cases}
    \end{equation}
    Outside the union of these intervals, set $\Psi(t):=t$. The intervals $L_{r,s,\ell}$ are pairwise disjoint, and $\Psi$ agrees with the identity at their endpoints. Hence $\Psi$ is continuous on $(0,1]$. Moreover, $\Psi$ maps each $L_{r,s,\ell}$ onto itself, and all intervals in the $r$-th batch are contained in $D_r=[2^{-r},\,3\cdot2^{-r-1}]$. Since $\sup D_r\to0$ as $r\to\infty$, we have $\Psi(t)\to0$ as $t\to0$. Therefore, $\Psi$ is continuous on $[0,1]$.
    
    Moreover, $\Psi(z_{r,s,\ell})=z_{r,s,\ell}+2^{-r^3-s-1}=y_{r,s,\ell}$, while $\Psi$ fixes the two endpoints of $L_{r,s,\ell}$. Every linear piece of $\Psi$ has slope $1\pm2^{-s-1}$, which lies between $3/4$ and $5/4$. Hence $\Psi$ is an increasing bi-Lipschitz homeomorphism of $[0,1]$.

    Define
    \begin{equation}\label{eq:deformed-dyadic-partition}
        \pi^n := \Psi(\T^n), \qquad n\ge0.
    \end{equation}
    The sequence $\pi$ is refining, and every interval $I$ of $\pi^n$ satisfies
    \[
        \frac34\,2^{-n}\le |I|\le\frac54\,2^{-n}.
    \]
    Thus $\pi$ is balanced. Moreover, every interval has exactly two children. Since $\pi^n$ consists of $2^n$ intervals whose lengths sum to one, $|\pi^n|\ge2^{-n}$, and hence
    \begin{equation}\label{eq:deformed-mesh-ratios}
        \frac85\le\frac{|\pi^n|}{|\pi^{n+1}|}\le\frac{10}{3}, \qquad n\ge0.
    \end{equation}
    Consequently, $\pi$ is a balanced complete refining partition sequence.

    For fixed $r,s,\ell$, the points $z_{r,s,\ell}-2^{-r^3}$, $z_{r,s,\ell}$, $z_{r,s,\ell}+2^{-r^3}$ are three consecutive points of $\T^{r^3}$. Their images under $\Psi$ are $z_{r,s,\ell}-2^{-r^3}$, $y_{r,s,\ell}$, $z_{r,s,\ell}+2^{-r^3}$, respectively, and are three consecutive points of $\pi^{r^3}$. Since the supports of the tents are pairwise disjoint,
    \[
        x\big(z_{r,s,\ell}-2^{-r^3}\big)=0, \qquad x(y_{r,s,\ell})=2^{-\frac{r^3-2r}{p}}, \qquad x\big(z_{r,s,\ell}+2^{-r^3}\big)=0.
    \]
    Each peak therefore contributes
    \[
        2\Big(2^{-\frac{r^3-2r}{p}}\Big)^p = 2^{1-r^3+2r}
    \]
    to the $p$-th variation of $x$ along $\pi^{r^3}$. Since the $r$-th batch contains $r\,2^{r^3-2r}$ peaks,
    \[
        [x]_{\pi^{r^3}}^{(p)}(1) \ge r\,2^{r^3-2r} \, 2^{1-r^3+2r} = 2r \xlongrightarrow{r\to\infty}\infty.
    \]
    Thus $x\notin\mathcal X_\pi^p$.

    Finally, \eqref{eq:endpoint-function-holder} shows that $x$ has positive H\"older regularity, and both $\pi$ and $\sigma=\T$ satisfy the assumptions of Theorem~\ref{thm:main-intro}. Combining that theorem with \eqref{eq:endpoint-dyadic-membership} proves \eqref{eq:critical-space-noninvariance}.
\end{proof}

\begin{remark}\label{rem:sufficient-alignment}
    If two refining partition sequences $\pi$ and $\sigma$ mutually refine each other with uniformly bounded subdivision multiplicities, then $\mathcal X_\pi^p=\mathcal X_\sigma^p$ with equivalent norms for every $p\ge1$. Indeed, if $Q$ refines $P$ and divides each interval of $P$ into at most $K$ intervals, then
    \[
        [x]_P^{(p)}(T)\le K^{p-1}[x]_Q^{(p)}(T).
    \]
    Proposition~\ref{prop:critical-space-noninvariance} shows that the balancedness and mesh assumptions of Theorem~\ref{thm:main-intro} do not imply such a relation. In its construction, an interval of $\pi^{r^3}$ is divided into $2^{r+1}+1$ intervals by the first dyadic level which contains all its endpoints, and this number diverges as $r\to\infty$.
\end{remark}

\bigskip

\section{Classical \texorpdfstring{$p$}{p}-variation spaces and dyadic Schauder coefficients}   \label{sec:classical-pvar-spaces}

The equality $p_\T(x)=p_{\cvar}(x)$ in Theorem~\ref{thm:main-intro} allows us to relate the Banach spaces $\mathcal X_\T^p$ studied in \cite{das-kim2024} to the classical $p$-variation spaces $\mathcal C^{p\text{-}\cvar}$ defined below. We first establish function-space relations for an arbitrary partition sequence containing a regular skeleton, including a compact subcritical embedding into the vanishing $p$-variation space. We then specialize to the dyadic partition sequence and characterize the classical variation index in terms of the levelwise Faber--Schauder quantities. At a fixed exponent, this leads to a coefficient-space refinement of the subcritical inclusion and to a strict hierarchy between the resulting critical spaces.

Throughout this section, we work on the unit interval $[0,1]$. This entails no loss of generality, since the linear time change for a partition sequence $\pi$ of $[0,T]$
\[
    \widetilde\pi^n:=\{t/T:t\in\pi^n\}, \qquad \widetilde x(t):=x(Tt), \quad t\in[0,1],
\]
preserves the $p$-th variation and the classical $p$-variation:
\[
    [\widetilde x]_{\widetilde\pi^n}^{(p)}(1) = [x]_{\pi^n}^{(p)}(T), \qquad \|\widetilde x\|_{p\text{-}\cvar;[0,1]} = \|x\|_{p\text{-}\cvar;[0,T]}.
\]

For $p\ge1$, we recall the Banach space $\mathcal X_\pi^p$ in \eqref{eq:Xp}, equipped with the norm $\|x\|_\pi^{(p)}$ of \eqref{eq:Xp-norm}. We also consider the classical $p$-variation Banach space
\begin{equation}\label{eq:classical-pvar-space}
    \mathcal C^{p\text{-}\cvar}([0,1]) := \big\{x\in C^0([0,1]):\|x\|_{p\text{-}\cvar}<\infty\big\},
\end{equation}
equipped with the norm
\[
    \|x\|_{\mathcal C^{p\text{-}\cvar}} := |x(0)|+\|x\|_{p\text{-}\cvar}.
\]
Let $\mathrm{PL}([0,1])$ denote the space of continuous piecewise linear functions on $[0,1]$. We further define
\begin{equation}\label{eq:vanishing-pvar-space}
    \mathcal C^{0,p\text{-}\cvar}([0,1]) := \overline{\mathrm{PL}([0,1])}^{\,\|\cdot\|_{\mathcal C^{p\text{-}\cvar}}},
\end{equation}
the closed subspace of $\mathcal C^{p\text{-}\cvar}([0,1])$ consisting of paths which can be approximated by piecewise linear paths in the classical $p$-variation norm. This space $\mathcal C^{0,p\text{-}\cvar}([0,1])$ plays a central role in rough path theory; see, for instance, \cite{friz2010,FrizHairer}.

\medskip

\subsection{Function-space relations}

For two Banach spaces $B_1$ and $B_2$, we denote by
\begin{equation*}
    \mathcal L(B_1,B_2) :=\{A:B_1\to B_2:A\text{ is linear and bounded}\}
\end{equation*}
the space of bounded linear operators from $B_1$ to $B_2$. This space is equipped with the operator norm
\begin{equation*}
    \|A\|_{\mathcal L(B_1,B_2)}:=\sup_{\|x\|_{B_1}\le1}\|Ax\|_{B_2} = \sup_{x\in B_1\setminus\{0\}} \frac{\|Ax\|_{B_2}}{\|x\|_{B_1}}.
\end{equation*}
Under this norm, $\mathcal L(B_1,B_2)$ is a Banach space.

We first derive an inclusion result from Proposition~\ref{prop:subcritical}. We equip the intersection space
\[
    E_{\alpha,q}^\pi:=C^\alpha([0,1])\cap\mathcal X_\pi^q
\]
with the norm
\begin{equation}    \label{def:E-alpha-q-norm}
    \|x\|_{E_{\alpha,q}^\pi} := |x|_{C^\alpha}+\|x\|_\pi^{(q)}.
\end{equation}
The completeness argument in \cite[Proposition~5.2]{das-kim2024} applies verbatim to the present range of $\alpha$ and $q$, and shows that $E_{\alpha,q}^\pi$ is a Banach space under this norm.

\begin{proposition}[Subcritical approximation and compact embedding]\label{prop:subcritical-compact}
    Let $\pi\in\Pi([0,1])$ contain a regular skeleton $\rho=(\pi^{n_m})_{m\ge0}$, let $\alpha\in(0,1]$, and let $1<q<p<\infty$. Then 
    \begin{equation}\label{eq:subcritical-vanishing-pvar}
        E_{\alpha,q}^\pi \subset\mathcal C^{0,p\text{-}\cvar}([0,1]).
    \end{equation}
    Moreover, when the space $E_{\alpha,q}^\pi$ is equipped with the norm $\|x\|_{E_{\alpha,q}^\pi}$, the embedding in \eqref{eq:subcritical-vanishing-pvar} is continuous and compact.
\end{proposition}

\begin{proof}
    By \eqref{eq:block-modulus}, $\|u_m^\rho\|_\infty\le |x|_{C^\alpha}|\rho^m|^\alpha$.
    Applying \eqref{eq:subcritical-tail} with $r=p$ gives
    \begin{equation}\label{eq:subcritical-holder-tail}
        \|x-x_N^\rho\|_{p\text{-}\cvar} \le 2^{1+\frac{q-1}{p}} \Big(\sup_{m\ge0}M_m(\rho)\Big)^{\frac{q-1}{p}} \Big(\sup_{n\ge0}[x]_{\pi^n}^{(q)}(1)\Big)^{\frac1p} |x|_{C^\alpha}^{1-\frac qp} \sum_{m=N}^{\infty}|\rho^m|^{\alpha(1-\frac qp)}.
    \end{equation}
    Since $\rho$ is regular,
    \[
        \sum_{m=0}^\infty|\rho^m|^{\alpha(1-\frac qp)}<\infty,
    \]
    and hence $x_N^\rho\to x$ as $N\to\infty$ in the classical $p$-variation norm. Since every $x_N^\rho$ is piecewise linear, this proves \eqref{eq:subcritical-vanishing-pvar}.

    Let
    \[
        \mathcal I:E_{\alpha,q}^\pi\longrightarrow\mathcal C^{0,p\text{-}\cvar}([0,1]), \qquad \mathcal I x:=x,
    \]
    denote the inclusion mapping. For every $N\ge0$, define
    \[
        P_N:E_{\alpha,q}^\pi\longrightarrow\mathcal C^{0,p\text{-}\cvar}([0,1]), \qquad P_Nx:=x_N^\rho.
    \]
    If $\rho^N=\{0=t_0^N<t_1^N<\cdots<t_{M_N}^N=1\}$, then the range of $P_N$ is contained in the finite-dimensional space
    \[
        V_N^\rho := \{f\in C^0([0,1]):f\text{ is linear on every interval of }\rho^N\}.
    \]
    Moreover, $P_Nx$ is determined by the finitely many continuous evaluation mappings
    \[
        x\longmapsto x(t_0^N),\ldots,x(t_{M_N}^N),
    \]
    so $P_N$ is a bounded finite-rank operator and hence compact. Since $x_N^\rho(0)=x(0)$, one has
    \[
        \|(\mathcal I-P_N)x\|_{\mathcal C^{0,p\text{-}\cvar}} = \|x-x_N^\rho\|_{p\text{-}\cvar}.
    \]
    Set
    \[
        |x|^{(q)}_{\pi} := \Big(\sup_{n\ge0}[x]_{\pi^n}^{(q)}(1)\Big)^{\frac1q}.
    \]
    Since $|x|^{(q)}_{\pi}\le\|x\|_\pi^{(q)}$, the weighted arithmetic--geometric mean inequality with \eqref{def:E-alpha-q-norm} gives
    \[
        (|x|^{(q)}_{\pi})^{\frac qp}|x|_{C^\alpha}^{1-\frac qp}
        \le\frac qp |x|^{(q)}_{\pi} + \Big(1-\frac qp\Big)|x|_{C^\alpha}
        \le\|x\|_{E_{\alpha,q}^\pi}.
    \]
    Hence, the inequality \eqref{eq:subcritical-holder-tail} yields the following bound for the operator norm
    \[
        \|\mathcal I-P_N\|_{\mathcal L(E_{\alpha,q}^\pi,\mathcal C^{0,p\text{-}\cvar})}
        \le 2^{1+\frac{q-1}{p}} \Big(\sup_{m\ge0}M_m(\rho)\Big)^{\frac{q-1}{p}} \sum_{m=N}^\infty|\rho^m|^{\alpha(1-\frac qp)}.
    \]
    The right-hand side converges to zero because $\rho$ is regular. Thus, $P_N\to\mathcal I$ as $N\to\infty$ in operator norm. Since every $P_N$ is compact, the inclusion mapping $\mathcal I$ is compact, and in particular continuous.
\end{proof}

Since each $\pi^n$ is one of the finite partitions appearing in \eqref{eq:classical-pvar}, one has $[x]_{\pi^n}^{(p)}(1)\le\|x\|_{p\text{-}\cvar}^p$ for every $n\ge0$. Moreover, $\|x\|_\infty\le|x(0)|+\|x\|_{p\text{-}\cvar}$. Hence, the following inclusion holds without any additional assumption on the partition sequence or regularity assumption on the function.

\begin{proposition}\label{prop:classical-into-partition}
    Let $\pi\in\Pi([0,1])$ and $p\ge1$. Then the inclusion
    \begin{equation}\label{eq:classical-into-partition}
        \mathcal C^{p\text{-}\cvar}([0,1])\subset\mathcal X_\pi^p
    \end{equation}
    is continuous. More precisely, $\|x\|_\pi^{(p)} \le\|x\|_{\mathcal C^{p\text{-}\cvar}}$ for $x\in\mathcal C^{p\text{-}\cvar}([0,1])$.
\end{proposition}

For $p>1$ and $\pi\in\Pi([0,1])$, introduce the subcritical and supercritical spaces
\begin{equation}\label{eq:pminus-pplus-spaces}
    \mathcal X_\pi^{p-}:=\bigcup_{1<q<p}\mathcal X_\pi^q,\qquad
    \mathcal X_\pi^{p+}:=\bigcap_{q>p}\mathcal X_\pi^q,\qquad
    \mathcal C_{\cvar}^{p+}:=\bigcap_{q>p}\mathcal C^{q\text{-}\cvar}([0,1]).
\end{equation}

Combining the previous results, we obtain the following inclusion chain.

\begin{corollary}[Subcritical and supercritical relations]\label{cor:subcritical-space-relations}
    Let $\pi\in\Pi([0,1])$ contain a regular skeleton, and let $\alpha\in(0,1]$ and $p>1$. Then
    \begin{equation}\label{eq:subcritical-space-sandwich}
        C^\alpha([0,1])\cap\mathcal X_\pi^{p-} \subset C^\alpha([0,1])\cap\mathcal C^{0,p\text{-}\cvar}([0,1]) \subset C^\alpha([0,1])\cap\mathcal C^{p\text{-}\cvar}([0,1]) \subset C^\alpha([0,1])\cap\mathcal X_\pi^p.
    \end{equation}
    Moreover,
    \begin{equation}\label{eq:supercritical-space-equality}
        C^\alpha([0,1])\cap\mathcal X_\pi^{p+} = C^\alpha([0,1])\cap\mathcal C_{\cvar}^{p+}.
    \end{equation}
\end{corollary}

\begin{proof}
    Suppose that $x\in C^\alpha([0,1])\cap\mathcal X_\pi^{p-}$. Then $x\in\mathcal X_\pi^q$ for some $1<q<p$. Since $\pi$ contains a regular skeleton, Proposition~\ref{prop:subcritical-compact} gives $x\in\mathcal C^{0,p\text{-}\cvar}([0,1])$. The second inclusion in \eqref{eq:subcritical-space-sandwich} follows from the definition of $\mathcal C^{0,p\text{-}\cvar}([0,1])$, and the third follows from Proposition~\ref{prop:classical-into-partition}.

    Next, the threshold property \eqref{eq:index-threshold} implies $x\in\mathcal X_\pi^{p+}$ if and only if $p_\pi(x)\le p$; similarly, $x\in\mathcal C_{\cvar}^{p+}$ if and only if $p_{\cvar}(x)\le p$. Then Corollary~\ref{cor:holder-skeleton} proves \eqref{eq:supercritical-space-equality}.
\end{proof}

\begin{remark}[The H\"older threshold $1/p$]    \label{rem:holder-threshold}
    Recall that every $x\in C^\alpha([0,1])$ has finite classical $r$-variation for every $r\ge1/\alpha$. Hence, if $\alpha>1/p$, choose $q$ such that $\max \{1, 1/\alpha\}<q<p$. Then
    \[
        C^\alpha([0,1]) \subset\mathcal C^{q\text{-}\cvar}([0,1]) \subset\mathcal X_\pi^q \subset\mathcal X_\pi^{p-}.
    \]
    Proposition~\ref{prop:subcritical-compact} also gives $C^\alpha([0,1]) \subset\mathcal C^{0,p\text{-}\cvar}([0,1])$. Consequently, all four spaces in \eqref{eq:subcritical-space-sandwich} coincide with $C^\alpha([0,1])$. Moreover, both sides of \eqref{eq:supercritical-space-equality} are also equal to $C^\alpha([0,1])$.

    At the critical H\"older exponent $\alpha=1/p$, one still has $C^{1/p}([0,1]) \subset\mathcal C^{p\text{-}\cvar}([0,1]) \subset\mathcal X_\pi^p$. Hence, the last two spaces of \eqref{eq:subcritical-space-sandwich} collapse to $C^{1/p}([0,1])$ and the sandwich takes the form
    \[
        C^{\frac1p}([0,1])\cap\mathcal X_\pi^{p-} \subset C^{\frac1p}([0,1])\cap\mathcal C^{0,p\text{-}\cvar}([0,1]) \subset C^{\frac1p}([0,1]).
    \]
    The two spaces on the left may still be proper subspaces of $C^{1/p}([0,1])$; thus, unlike the case $\alpha>1/p$, nontrivial endpoint phenomena remain. Both spaces of \eqref{eq:supercritical-space-equality} also collapse to $C^{1/p}([0,1])$. Thus, the genuinely nontrivial range of \eqref{eq:subcritical-space-sandwich} is $0<\alpha\le1/p$, with $\alpha=1/p$ representing the critical H\"older endpoint.
\end{remark}

The inclusions in \eqref{eq:subcritical-space-sandwich} leave open only the critical endpoint. Equality of the variation indices does not determine membership at that endpoint. In fact, the path constructed in Proposition~\ref{prop:critical-space-noninvariance} satisfies
\[
    p_{\cvar}(x)=p, \qquad x\notin\mathcal C^{p\text{-}\cvar}([0,1]).
\]
Indeed, otherwise Proposition~\ref{prop:classical-into-partition} would imply $x\in\mathcal X_\pi^p$, contrary to \eqref{eq:critical-space-noninvariance}.

\medskip

\subsection{A Faber--Schauder characterization of the classical variation index}

The results in the previous subsection are formulated for an arbitrary partition sequence $\pi$ containing a regular skeleton. From this subsection onward, we specialize to the dyadic partition sequence $\T$, for which the levelwise Faber--Schauder quantities take their simplest form. Under the stronger assumptions imposed in \cite[Theorem~4.3]{das-kim2024}, analogous coefficient formulas are available for more general partition sequences. We restrict attention to $\T$ in order to avoid introducing a second family of partition-dependent coefficient notation.

Recall the dyadic Faber--Schauder functions $(e_{m,k}^\T)_{m\ge0,\,k\in I_m}$ from \eqref{Eq: e_mk} on $[0,1]$ with the index set $I_m:=\{0,1,\ldots,2^m-1\}$, and the corresponding Faber--Schauder coefficients $(\theta_{m,k}^{x,\T})_{m,k}$ of \eqref{eq:theta_mk} from the representation \eqref{Eq: Schauder representation} of $x\in C^0([0,1])$. For $q>1$, recall from \eqref{eq:dyadic-index-coeff} the levelwise Schauder quantity
\begin{equation}\label{eq:dyadic-xi}
    \xi_n^{(q)}(x) := 2^{-\frac{nq}{2}} \sum_{k\in I_n}|\theta_{n,k}^{x,\T}|^q, \qquad n\ge0.
\end{equation}
For the dyadic partition sequence $\T$, the results of \cite[Theorem~4.3 and Corollary~4.4]{das-kim2024} yield for any $q>1$
\begin{equation}\label{eq:dyadic-Xp-xi}
    x\in\mathcal X_\T^q \quad \text{if and only if} \quad \limsup_{n\to\infty}\xi_n^{(q)}(x)<\infty
\end{equation}
and
\begin{equation}\label{eq:dyadic-index-xi}
    p_\T(x)=\inf\Big\{q>1:\limsup_{n\to\infty}\xi_n^{(q)}(x)<\infty\Big\}.
\end{equation}
Since each $\xi_n^{(q)}(x)$ is finite, the condition $\limsup_{n\to\infty}\xi_n^{(q)}(x)<\infty$ is equivalent to $\sup_{n\ge0}\xi_n^{(q)}(x)<\infty$. Combining \eqref{eq:dyadic-index-xi} with Corollary~\ref{cor:dini-main} gives the following characterization of the classical $p$-variation index.

\begin{corollary}[Faber--Schauder characterization of the classical variation index]\label{cor:classical-index-schauder}
    Let $x\in C^0([0,1])$ have an all-orders Dini modulus. Then
    \[
        p_{\cvar}(x) = \inf \Big\{q>1:\limsup_{n\to\infty}\xi_n^{(q)}(x) < \infty\Big\}.
    \]
    Moreover, for every $p>1$,
    \[
        x\in\mathcal C_{\cvar}^{p+} \quad \iff \quad p_{\cvar}(x)\le p \quad \iff \quad \limsup_{n\to\infty}\xi_n^{(q)}(x)<\infty \quad\text{for every }q>p.
    \]
\end{corollary}

\medskip

\subsection{Coefficient conditions at the critical endpoint}\label{subsec:critical-coefficient-conditions}

We now turn from the location of the critical exponent to membership at a fixed exponent $p$. Uniform boundedness of the levelwise quantities $\xi_n^{(p)}$ is necessary for finite classical $p$-variation, whereas summability of the level norms $(\xi_n^{(p)}(x))^{1/p}$ is sufficient. For $p>1$, define the subspace
\begin{equation}\label{eq:outer-l1-coefficient-space}
    \mathcal S_\T^{p} := \Big\{x\in C^0([0,1]): \sum_{n=0}^\infty\big(\xi_n^{(p)}(x)\big)^{\frac1p}<\infty\Big\},
\end{equation}
with the norm
\begin{equation}\label{eq:S-norm}
    \|x\|_{\mathcal S_\T^p} := |x(0)|+|x(1)-x(0)| + \sum_{n=0}^\infty\big(\xi_n^{(p)}(x)\big)^{\frac1p}.
\end{equation}

\begin{proposition}[Banach structure of the coefficient space]\label{prop:S-Banach}
    For every $p>1$, the space $\mathcal S_\T^p$ is a separable Banach space under the norm \eqref{eq:S-norm}. Moreover, for every $x\in\mathcal S_\T^p$, if $x_N^\T$ denotes its piecewise linear approximation along $\T^N$, then
    \[
        \|x-x_N^\T\|_{\mathcal S_\T^p}
        =\sum_{n=N}^\infty\big(\xi_n^{(p)}(x)\big)^{\frac1p}
        \xlongrightarrow{N\to\infty}0.
    \]
\end{proposition}

\begin{proof}
    For $x\in\mathcal S_\T^p$, set
    \[
        a_n^x:=\big(2^{-\frac n2}\theta_{n,k}^{x,\T}\big)_{k\in I_n}, \qquad n\ge0.
    \]
    Then $\|a_n^x\|_{\ell^p}=(\xi_n^{(p)}(x))^{1/p}$. Let us consider a space of real sequences
    \[
        \mathfrak s_p := \Big\{a=(a_n)_{n\ge0}:a_n\in\R^{I_n},\ \sum_{n=0}^\infty\|a_n\|_{\ell^p}<\infty\Big\},
    \]
    equipped with the norm $\|a\|_{\mathfrak s_p}:=\sum_{n\ge0}\|a_n\|_{\ell^p}$. This is the $\ell^1$-direct sum of the finite-dimensional Banach spaces $\ell^p(I_n)$ and hence is a Banach space.

    Consider the coefficient mapping
    \[
        \Gamma_p x := \big(x(0), \, x(1)-x(0), \, (a_n^x)_{n\ge0}\big).
    \]
    Equip $\R^2$ with the norm $(a,b)\mapsto|a|+|b|$. By \eqref{eq:S-norm}, $\Gamma_p$ is a linear isometry from $\mathcal S_\T^p$ into $\R^2\oplus_1\mathfrak s_p$. To see that it is onto, let $a,b\in\R$ and $c=(c_n)_{n\ge0}\in\mathfrak s_p$, where $c_n=(c_{n,k})_{k\in I_n}$, and define
    \[
        x(t):=a+bt+\sum_{n=0}^\infty\sum_{k\in I_n}2^{\frac n2}c_{n,k}e_{n,k}^\T(t).
    \]
    Since the Faber--Schauder functions at each fixed level have disjoint interiors and $\|e_{n,k}^\T\|_\infty=2^{-n/2-1}$,
    \[
        \Big\|\sum_{k\in I_n}2^{\frac n2}c_{n,k}e_{n,k}^\T\Big\|_\infty
        \le\frac12\|c_n\|_{\ell^p}.
    \]
    Hence the series converges uniformly and defines a continuous function with the prescribed coefficients. Thus, $\Gamma_p$ is an isometric isomorphism, which proves completeness. Since $\mathfrak s_p$ is the countable $\ell^1$-direct sum of finite-dimensional spaces, it is separable. Hence $\mathbb R^2\oplus_1\mathfrak s_p$, and therefore $\mathcal S_\T^p$, is separable. The assertion for $x_N^\T$ follows directly by removing the coefficient rows of levels $n\ge N$.
\end{proof}

Before establishing the endpoint estimate, we show that this summability condition follows from subcritical dyadic variation control under a summability condition on the modulus of continuity. In particular, this applies to paths with an all-orders Dini modulus, and hence to paths with positive H\"older regularity.

\begin{proposition}[Subcritical coefficient summability]\label{prop:subcritical-coefficient-summability}
    Let $1<q<p<\infty$ and $x\in\mathcal X_\T^q$. Suppose that
    \[
        \sum_{n=0}^\infty \omega_x\big(2^{-(n+1)}\big)^{1-\frac qp}<\infty.
    \]
    Then $x\in\mathcal S_\T^p$. The summability assumption is satisfied whenever $x$ has an all-orders Dini modulus, and hence whenever $x\in C^\alpha([0,1])$ for some $\alpha\in(0,1]$.
\end{proposition}

\begin{proof}
    The coefficient formula \eqref{eq:theta_mk} gives
    \[
        2^{-\frac n2}\max_{k\in I_n} |\theta_{n,k}^{x,\T}| \le 2\omega_x\big(2^{-(n+1)}\big).
    \]
    By interpolation between the $\ell^q$- and $\ell^\infty$-norms,
    \begin{align*}
        \big(\xi_n^{(p)}(x)\big)^{\frac1p}
        &= 2^{-\frac n2} \Big(\sum_{k\in I_n}|\theta_{n,k}^{x,\T}|^p\Big)^{\frac1p}
        \le \bigg[ 2^{-\frac n2} \Big(\sum_{k\in I_n}|\theta_{n,k}^{x,\T}|^q\Big)^{\frac1q} \bigg]^{\frac qp} \bigg[ 2^{-\frac n2}\max_{k\in I_n} |\theta_{n,k}^{x,\T}| \bigg]^{1-\frac qp} \\
        &\le 2^{1-\frac qp} \big(\xi_n^{(q)}(x)\big)^{\frac1p} \, \omega_x\big(2^{-(n+1)}\big)^{1-\frac qp}
        \le 2^{1-\frac qp} \Big(\sup_{m\ge0}\xi_m^{(q)}(x)\Big)^{\frac1p} \, \omega_x\big(2^{-(n+1)}\big)^{1-\frac qp}.
    \end{align*}
    Since $x\in\mathcal X_\T^q$, the characterization \eqref{eq:dyadic-Xp-xi} gives $\sup_{m\ge0}\xi_m^{(q)}(x)<\infty$. Summing the preceding estimate over $n$ proves that
    \[
        \sum_{n=0}^\infty \big(\xi_n^{(p)}(x)\big)^{\frac1p}<\infty,
    \]
    and hence $x\in\mathcal S_\T^p$.

    Suppose next that $x$ has an all-orders Dini modulus. Since $\omega_x$ is nondecreasing, for every $\delta>0$,
    \[
        \omega_x\big(2^{-(n+1)}\big)^\delta\log2 \le \int_{2^{-(n+1)}}^{2^{-n}} \omega_x(h)^\delta\,\frac{dh}{h}.
    \]
    Consequently,
    \[
        \sum_{n=0}^\infty \omega_x\big(2^{-(n+1)}\big)^\delta \le \frac1{\log2} \int_0^1\omega_x(h)^\delta\,\frac{dh}{h}<\infty.
    \]
    Taking $\delta=1-q/p$ verifies the required summability condition.
\end{proof}

\begin{proposition}[Endpoint coefficient estimates] \label{prop:endpoint-coefficient-estimates}
    Let $p>1$. If $x\in \mathcal S_\T^{p}$, then for every $N\ge0$,
    \begin{equation}\label{eq:endpoint-tail}
        \|x-x_N^\T\|_{p\text{-}\cvar} \le \sum_{n=N}^\infty\big(\xi_n^{(p)}(x)\big)^{\frac1p}.
    \end{equation}
    Consequently, $x\in\mathcal C^{0,p\text{-}\cvar}([0,1])$ and
    \begin{equation}\label{eq:endpoint-sufficient}
        \|x\|_{p\text{-}\cvar} \le |x(1)-x(0)|+\sum_{n=0}^\infty\big(\xi_n^{(p)}(x)\big)^{\frac1p}.
    \end{equation}
    Conversely, if $x\in\mathcal C^{p\text{-}\cvar}([0,1])$, then
    \begin{equation}\label{eq:endpoint-necessary}
        \sup_{n\ge0}\big(\xi_n^{(p)}(x)\big)^{\frac1p} \le 2^{1-\frac1p}\|x\|_{p\text{-}\cvar}.
    \end{equation}
\end{proposition}

\begin{proof}
    Since $x_n^\T$ is the piecewise linear approximation of $x$ along $\T^n$, we can write
    \[
        u_n^\T:=x_{n+1}^\T-x_n^\T = \sum_{k\in I_n}\theta_{n,k}^{x,\T}e_{n,k}^\T.
    \]
    Given a finite partition $\mathfrak P$, set $\widehat{\mathfrak P}:=\mathfrak P\cup\T^n$. For every interval $[s,t]$ of $\mathfrak P$ containing a point of $\T^n$ in its interior, all newly inserted points are zeros of $u_n^\T$. Hence
    \[
        |u_n^\T(t)-u_n^\T(s)|^p \le 2^{p-1}\big(|u_n^\T(s)|^p+|u_n^\T(t)|^p\big).
    \]
    If no point is inserted into $[s,t]$, its increment remains unchanged. Summing over the intervals of $\mathfrak P$ gives
    \[
        [u_n^\T]_{\mathfrak P}^{(p)}(1) \le 2^{p-1}[u_n^\T]_{\widehat{\mathfrak P}}^{(p)}(1).
    \]
    On the dyadic interval $I_{n,k}:=[t_k^n,t_{k+1}^n]$, the function $u_n^\T$ is a single tent of absolute height $H_{n,k}:=2^{-\frac n2-1}|\theta_{n,k}^{x,\T}|$. Since the total variation of $u_n^\T$ on $I_{n,k}$ is equal to $2H_{n,k}$ and $|u_n^\T(t)-u_n^\T(s)|\le H_{n,k}$ for all $s,t\in I_{n,k}$, we have
    \begin{align*}
        \sum_{\substack{[s,t]\in\widehat{\mathfrak P} \\ [s,t]\subset I_{n,k}}} |u_n^\T(t)-u_n^\T(s)|^p \le H_{n,k}^{p-1} \sum_{\substack{[s,t]\in\widehat{\mathfrak P}\\ [s,t]\subset I_{n,k}}} |u_n^\T(t)-u_n^\T(s)| \le 2H_{n,k}^p.
    \end{align*}
    Therefore,
    \[
        [u_n^\T]_{\mathfrak P}^{(p)}(1) \le 2^{p-1}\sum_{k\in I_n}2 \Big(2^{-\frac n2-1}|\theta_{n,k}^{x,\T}|\Big)^p = \xi_n^{(p)}(x),
    \]
    and hence
    \begin{equation}\label{eq:level-block-pvar}
        \|u_n^\T\|_{p\text{-}\cvar} \le \big(\xi_n^{(p)}(x)\big)^{\frac1p}.
    \end{equation}
    For integers $K>N$, the triangle inequality and \eqref{eq:level-block-pvar} give
    \[
        \|x_K^\T-x_N^\T\|_{p\text{-}\cvar} \le \sum_{n=N}^{K-1}\big(\xi_n^{(p)}(x)\big)^{\frac1p}.
    \]
    Letting $K\to\infty$ and using the lower semicontinuity of the classical $p$-variation from Lemma~\ref{lem:pvar-lower-semicontinuity} proves \eqref{eq:endpoint-tail}. Since each $x_N^\T$ is piecewise linear, it follows that $x\in\mathcal C^{0,p\text{-}\cvar}([0,1])$. Taking $N=0$ and using $\|x_0^\T\|_{p\text{-}\cvar}=|x(1)-x(0)|$ gives \eqref{eq:endpoint-sufficient}.

    For the converse, let
    \[
        a_{n,k}:=x(t_{2k+1}^{n+1})-x(t_{2k}^{n+1}),\qquad
        b_{n,k}:=x(t_{2k+2}^{n+1})-x(t_{2k+1}^{n+1}).
    \]
    The coefficient formula \eqref{eq:theta_mk} for the dyadic Schauder coefficients gives $2^{-\frac{n}{2}} \theta_{n,k}^{x,\T}=a_{n,k}-b_{n,k}$. Consequently,
    \[
        \xi_n^{(p)}(x) = \sum_{k\in I_n}|a_{n,k}-b_{n,k}|^p \le 2^{p-1}\sum_{k\in I_n}\big(|a_{n,k}|^p+|b_{n,k}|^p\big) = 2^{p-1}[x]_{\T^{n+1}}^{(p)}(1).
    \]
    Taking the supremum over $n$ proves \eqref{eq:endpoint-necessary}.
\end{proof}

The same level norms $(\xi_n^{(p)}(x))^{1/p}$ also appear in the Faber--Schauder energy spaces of \cite{kim2026}, where different cross-level summability and tail conditions are used to construct a pathwise integral beyond the classical Young regime. Here, their $\ell^1$-summability instead controls approximation in the classical $p$-variation norm.

Combining Propositions~\ref{prop:subcritical-coefficient-summability}, \ref{prop:endpoint-coefficient-estimates}, and \ref{prop:classical-into-partition} yields the following coefficient-space relations.

\begin{corollary}[Coefficient-space inclusions]\label{cor:sandwich}
    Let $p>1$. Then
    \begin{equation}\label{eq:endpoint-coefficient-sandwich}
        \mathcal S_\T^p \subset\mathcal C^{0,p\text{-}\cvar}([0,1]) \subset\mathcal C^{p\text{-}\cvar}([0,1]) \subset\mathcal X_\T^p.
    \end{equation}
    Moreover, for every $\alpha\in(0,1]$,
    \begin{equation}\label{eq:subcritical-coefficient-refinement}
        C^\alpha([0,1])\cap\mathcal X_\T^{p-} \subset\mathcal S_\T^p.
    \end{equation}
    More generally, every path with an all-orders Dini modulus which belongs to $\mathcal X_\T^{p-}$ also belongs to $\mathcal S_\T^p$.
\end{corollary}

\begin{proof}
    The first inclusion in \eqref{eq:endpoint-coefficient-sandwich} follows from Proposition~\ref{prop:endpoint-coefficient-estimates}, the second follows from the definition of $\mathcal C^{0,p\text{-}\cvar}([0,1])$, and the third follows from Proposition~\ref{prop:classical-into-partition} applied to $\T$.

    Let $x\in C^\alpha([0,1])\cap\mathcal X_\T^{p-}$. Then $x\in\mathcal X_\T^q$ for some $1<q<p$. Proposition~\ref{prop:subcritical-coefficient-summability} gives $x\in\mathcal S_\T^p$, proving \eqref{eq:subcritical-coefficient-refinement}. The conclusion for a path with an all-orders Dini modulus also follows from Proposition~\ref{prop:subcritical-coefficient-summability}.
\end{proof}

\begin{remark}[A Faber--Schauder formulation of the coefficient sandwich]
    \label{rem:coefficient-formulation-sandwich}
    Let $\alpha\in(0,1)$ and define
    \[
        \mathcal H_\alpha^\T(x) := \sup_{n\ge0}\sup_{k\in I_n} 2^{n(\alpha-\frac12)}|\theta_{n,k}^{x,\T}|.
    \]
    Since $|\T^{n+1}|=2^{-(n+1)}$, the dyadic Ciesielski characterization \cite[Theorem~3.4]{fake_fBM} gives
    \[
        x\in C^\alpha([0,1]) \quad \Longleftrightarrow \quad \mathcal H_\alpha^\T(x)<\infty.
    \]
    Together with \eqref{eq:dyadic-Xp-xi}, Corollary~\ref{cor:sandwich} therefore yields, for every $p>1$,
    \begin{align*}
        &\Big\{x\in C^0([0,1]):\mathcal H_\alpha^\T(x)<\infty,\ 
        \sup_{n\ge0}\xi_n^{(q)}(x)<\infty
        \text{ for some }1<q<p\Big\} = C^\alpha([0,1])\cap\mathcal X_\T^{p-} \\
        & \qquad \qquad \subset\mathcal S_\T^p \subset\mathcal C^{0,p\text{-}\cvar}([0,1]) \subset\mathcal C^{p\text{-}\cvar}([0,1]) \subset\mathcal X_\T^p = \Big\{x\in C^0([0,1]): \sup_{n\ge0}\xi_n^{(p)}(x)<\infty\Big\}.
    \end{align*}
    Thus, the strictly subcritical class on the left, the space $\mathcal S_\T^p$, and the ambient space $\mathcal X_\T^p$ are all described entirely in terms of the dyadic Faber--Schauder coefficients. The two classical variation spaces are sandwiched between these coefficient-defined classes.
\end{remark}

For a fixed path $x$, the rightmost coefficient condition $\sup_{n\ge0}\xi_n^{(q)}(x)<\infty$ in Remark~\ref{rem:coefficient-formulation-sandwich} holds for every $q>p_\T(x)$ and fails for every $1<q<p_\T(x)$. If $1<p_\T(x)<\infty$, either possibility may occur at the critical exponent $q=p_\T(x)$. For positively H\"older continuous paths, Theorem~\ref{thm:main-intro} identifies this threshold with $p_{\cvar}(x)$. On the other hand, the H\"older assertion in Corollary~\ref{cor:sandwich} gives
\[
    \big\{x\in C^\alpha([0,1]):p_{\cvar}(x)<p\big\} \subset\mathcal S_\T^p.
\]
Consequently,
\[
    \Big(C^\alpha([0,1])\cap\mathcal C^{p\text{-}\cvar}([0,1])\Big) \setminus\mathcal S_\T^p \subset \big\{x\in C^\alpha([0,1]):p_{\cvar}(x)=p\big\}.
\]
Thus, within every positive H\"older class, the failure of $\mathcal S_\T^p$ to characterize finite classical $p$-variation is confined entirely to the critical boundary $p_{\cvar}(x)=p$. The next subsection shows that the universal inclusions in \eqref{eq:endpoint-coefficient-sandwich} are nevertheless all strict.

\medskip

\subsection{Strictness of the critical endpoint inclusions} \label{subsec:strict-endpoint-hierarchy}

We first provide a necessary condition for approximation by piecewise linear paths in the classical $p$-variation norm. We recall the notation \eqref{eq:dyadic-index-coeff} in terms of Faber--Schauder coefficients.

\begin{lemma}[Vanishing level energies]\label{lem:vanishing-level-energies}
    Let $p>1$. If $x\in\mathcal C^{0,p\text{-}\cvar}([0,1])$, then
    \begin{equation}\label{eq:vanishing-level-energies}
        \xi_n^{(p)}(x)\longrightarrow0
        \qquad\text{as }n\to\infty.
    \end{equation}
\end{lemma}

\begin{proof}
    Let $y$ be a continuous piecewise linear function, and let $J_y$ denote the number of its breakpoints in $(0,1)$. If $\theta_{n,k}^{y,\T}\ne0$, then the interior of the support of $e_{n,k}^\T$ contains at least one breakpoint of $y$. Since the interiors of the level-$n$ supports are pairwise disjoint,
    \[
        \#\big\{k\in I_n:\theta_{n,k}^{y,\T}\ne0\big\}\le J_y.
    \]
    If $L_y$ denotes the maximum absolute slope of $y$, then the coefficient formula \eqref{eq:theta_mk} gives
    \[
        2^{-\frac n2}|\theta_{n,k}^{y,\T}| \le L_y2^{-n}.
    \]
    Consequently, we have
    \begin{equation}    \label{ineq:JyLy}
        \xi_n^{(p)}(y) \le J_yL_y^p2^{-np},
    \end{equation}
    hence $\xi_n^{(p)}(y)\to0$ as $n\to\infty$ and $y\in\mathcal S_\T^p$.

    Now let $x\in\mathcal C^{0,p\text{-}\cvar}([0,1])$. Choose piecewise linear functions $y_j$ such that $\|x-y_j\|_{p\text{-}\cvar}\to0$ as $j\to\infty$. By the linearity of the Schauder coefficients and Minkowski's inequality,
    \[
        \big(\xi_n^{(p)}(x)\big)^{\frac1p} \le \big(\xi_n^{(p)}(x-y_j)\big)^{\frac1p} + \big(\xi_n^{(p)}(y_j)\big)^{\frac1p}.
    \]
    The last assertion of Proposition~\ref{prop:endpoint-coefficient-estimates} gives
    \[
        \sup_{n\ge0}\big(\xi_n^{(p)}(x-y_j)\big)^{\frac1p} \le 2^{1-\frac1p}\|x-y_j\|_{p\text{-}\cvar}.
    \]
    Taking first $\limsup_{n\to\infty}$ and then $j\to\infty$ proves \eqref{eq:vanishing-level-energies}.
\end{proof}

\begin{theorem}[Strict endpoint hierarchy]\label{thm:strict-endpoint-hierarchy}
    For every $p>1$, all inclusions in \eqref{eq:endpoint-coefficient-sandwich} are strict:
    \begin{equation}\label{eq:strict-endpoint-hierarchy}
        \mathcal S_\T^{p}
        \subsetneq \mathcal C^{0,p\text{-}\cvar}([0,1])
        \subsetneq \mathcal C^{p\text{-}\cvar}([0,1])
        \subsetneq \mathcal X_\T^p.
    \end{equation}
\end{theorem}

\begin{proof}
    We use the dyadic tents $\varphi_{n,k}$ of \eqref{eq:normalized-tent} and the triangular waves $\Phi_n$ of \eqref{eq:triangular-wave}.

    First, fix $\gamma\in(1/p,1]$ and define
    \begin{equation}    \label{def:f-gamma}
        f_\gamma(t):=\sum_{n=1}^\infty n^{-\gamma}\varphi_{n,1}(t), \qquad t \in [0, 1].
    \end{equation}
    The supports of the functions $\varphi_{n,1}$ have disjoint interiors and accumulate only at the origin. If
    \[
        f_\gamma^{(N)}:=\sum_{n=1}^N n^{-\gamma}\varphi_{n,1},
    \]
    then
    \[
        \|f_\gamma-f_\gamma^{(N)}\|_{p\text{-}\cvar}^p = 2\sum_{n>N}n^{-\gamma p} \xlongrightarrow{N\to\infty} 0.
    \]
    Hence, $f_\gamma\in\mathcal C^{0,p\text{-}\cvar}([0,1])$. On the other hand, $\xi_n^{(p)}(f_\gamma)=2^pn^{-\gamma p}$ and therefore
    \[
        \sum_{n=1}^\infty\big(\xi_n^{(p)}(f_\gamma)\big)^{\frac1p} = 2\sum_{n=1}^\infty n^{-\gamma}=\infty.
    \]
    Thus, the first inclusion is strict.

    Next, define
    \[
        F_p(t):=\sum_{n=0}^\infty2^{-\frac np}\Phi_n(t), \qquad t \in [0, 1].
    \]
    By \eqref{eq:wave-holder}, $F_p\in C^{1/p}([0,1])$, and hence $F_p\in\mathcal C^{p\text{-}\cvar}([0,1])$. However, \eqref{eq:xi-wave} gives $\xi_n^{(p)}(F_p)=2^p$ for every $n\ge0$. Lemma~\ref{lem:vanishing-level-energies} therefore shows that $F_p\notin\mathcal C^{0,p\text{-}\cvar}([0,1])$. Thus the second inclusion is strict.

    Finally, set $\gamma = 1/p$ in \eqref{def:f-gamma}:
    \begin{equation}    \label{def:f-1-p}
        f_{1/p}(t) = \sum_{n=1}^\infty n^{-\frac1p}\varphi_{n,1}(t), \qquad t \in [0, 1].
    \end{equation}
    The function $f_{1/p}$ is continuous and $\xi_n^{(p)}(f_{1/p})=2^p n^{-1}$ for $n\ge1$. Hence, $f_{1/p}\in\mathcal X_\T^p$ by \eqref{eq:dyadic-Xp-xi}. Since the tents are nonnegative, have disjoint interiors, and vanish at the endpoints of their supports, for every $r\ge1$,
    \begin{equation}    \label{eq:f-r-var}
        \|f_{1/p}\|_{r\text{-}\cvar}^r = 2\sum_{n=1}^\infty n^{-\frac rp}.
    \end{equation}
    In particular, $\|f_{1/p}\|_{p\text{-}\cvar}^p = \infty$. Therefore $f_{1/p}\notin\mathcal C^{p\text{-}\cvar}([0,1])$, and the third inclusion is strict.
\end{proof}

From the estimate \eqref{ineq:JyLy} in the proof of Lemma~\ref{lem:vanishing-level-energies}, every continuous piecewise linear function belongs to $\mathcal S_\T^p$. Since $\mathcal C^{0,p\text{-}\cvar}([0,1])$ is the closure of the piecewise linear functions in the classical $p$-variation norm, the continuous embedding
\[
    \mathcal S_\T^p\hookrightarrow\mathcal C^{0,p\text{-}\cvar}([0,1])
\]
has dense range. The strictness of the first inclusion in \eqref{eq:strict-endpoint-hierarchy} therefore shows that this range is not closed.

\begin{remark}[Failure of index invariance without regularity]  \label{rem:function-regularity-sharpness}
    For the last example $f_{1/p}$ in \eqref{def:f-1-p}, we have $p_\T(f_{1/p})=1$ and $p_{\cvar}(f_{1/p})=p$. Indeed, for every $q>1$, one has $\xi_n^{(q)}(f_{1/p})=2^q n^{-q/p}$ for $n\ge1$, and this sequence is uniformly bounded, whereas \eqref{eq:f-r-var} shows that $\|f_{1/p}\|_{r\text{-}\cvar}<\infty$ if and only if $r>p$. Thus, some regularity assumption on the function, such as H\"older regularity or an all-orders Dini modulus, is indispensable for the equality $p_\T(x)=p_{\cvar}(x)$.
\end{remark}

The same example $f_{1/p}$ in the proof of Theorem~\ref{thm:strict-endpoint-hierarchy} also yields a functional-analytic consequence of the strict final inclusion in \eqref{eq:strict-endpoint-hierarchy}.

\begin{corollary}[Nonclosed range of the classical $p$-variation embedding] \label{cor:nonclosed-classical-embedding}
    For every $p>1$, the continuous embedding $\mathcal C^{p\text{-}\cvar}([0,1]) \hookrightarrow\mathcal X_\T^p$ does not have closed range.
\end{corollary}

\begin{proof}
    Recall the function $f_{1/p}$ from \eqref{def:f-1-p}. Let
    \[
        f_{1/p}^{(N)}:=\sum_{n=1}^N n^{-\frac1p}\varphi_{n,1}.
    \]
    Each $f_{1/p}^{(N)}$ is piecewise linear and hence belongs to $\mathcal C^{p\text{-}\cvar}([0,1])$. Set $R_N:=f_{1/p}-f_{1/p}^{(N)}$. If $m\le N+1$, then $R_N$ vanishes at every point of $\T^m$. If $m>N+1$, a direct calculation gives
    \[
        [R_N]_{\T^m}^{(p)}(1) = 2^p\sum_{n=N+1}^{m-1}\frac{2^{-(p-1)(m-n)}}{n} \le \frac{2^p}{N+1}\sum_{\ell=1}^\infty2^{-(p-1)\ell}.
    \]
    Moreover, the disjoint supports of the tents give $\|R_N\|_\infty=(N+1)^{-1/p}\to0$ as $N\to\infty$. Together with the preceding estimate, this proves $\|f_{1/p}-f_{1/p}^{(N)}\|_\T^{(p)}\to0$ as $N\to\infty$. Since $f_{1/p}\notin\mathcal C^{p\text{-}\cvar}([0,1])$ by Theorem~\ref{thm:strict-endpoint-hierarchy}, the range of the embedding is not closed. 
\end{proof}

\bigskip

\bigskip

\noindent \textbf{Acknowledgment}

\smallskip

\noindent The author is grateful to Purba Das for suggesting the question of partition invariance of the variation index, for helpful discussions, and for reading the draft.

\bigskip

\noindent \textbf{Funding}

\smallskip

\noindent Donghan Kim is supported by the National Research Foundation of Korea under Grant RS-2025-00513609, funded by the Korean government (MSIT), and by an Ewon Assistant Professorship Award at KAIST.

\bigskip

\bibliography{pathwise1}

@unpublished{kim2026,
  author = {Kim, Donghan},
  title = {{Pathwise integration beyond Young via Faber--Schauder energy spaces}},
  year = {2026},
  note = {Preprint, arXiv:2606.13331}
}

@unpublished{das-kim-lim2026,
  author = {Das, Purba and Kim, Donghan and Lim, Fang Rui},
  title = {Banach spaces of continuous paths with finite $p$-th variation},
  year = {2026},
  note = {Preprint, arXiv:2604.05941}
}

@article{das-kim2024,
title = {On isomorphism of the space of continuous functions with finite $p$-th variation along a partition sequence},
journal = {Journal de Mathématiques Pures et Appliquées},
volume = {203},
pages = {103753},
year = {2025},
issn = {0021-7824},
doi = {https://doi.org/10.1016/j.matpur.2025.103753},
url = {https://www.sciencedirect.com/science/article/pii/S0021782425000972},
author = {Purba Das and Donghan Kim}
}

@article{fake_fBM,
  author = {Bayraktar, Erhan and Das, Purba and Kim, Donghan},
  title = {H\"older regularity and roughness: construction and examples},
  year = {2025},
  volume = {31},
  issue = {2},
  pages = {1084-1113},
  doi = {10.3150/24-BEJ1761},
  journal = {Bernoulli}
}

@article {Ciesielski:isomorphism,
    AUTHOR = {Ciesielski, Z.},
    TITLE = {On the isomorphisms of the spaces {$H_{\alpha }$} and {$m$}},
    JOURNAL = {Bull. Acad. Polon. Sci. S\'{e}r. Sci. Math. Astronom. Phys.},
    FJOURNAL = {Bulletin de l'Acad\'{e}mie Polonaise des Sciences. S\'{e}rie des Sciences Math\'{e}matiques, Astronomiques et Physiques},
    VOLUME = {8},
    YEAR = {1960},
    PAGES = {217--222},
    ISSN = {0001-4117},
    MRCLASS = {46.25},
    MRNUMBER = {132389},
    MRREVIEWER = {L. Nachbin},
}

@article{das2020,
    author = {Rama Cont and Purba Das},
title = {{Quadratic variation and quadratic roughness}},
volume = {29},
journal = {Bernoulli},
number = {1},
publisher = {Bernoulli Society for Mathematical Statistics and Probability},
pages = {496 -- 522},
year = {2023}
}

@article{das2021,
 title = {Quadratic variation along refining partitions: Constructions and examples},
journal = {Journal of Mathematical Analysis and Applications},
volume = {512},
number = {2},
pages = {126173},
year = {2022},
issn = {0022-247X},
author = {Rama Cont and Purba Das}
}

@article{ananova2017,
	Author = {Anna Ananova and Rama Cont},
	Doi = {http://dx.doi.org/10.1016/j.matpur.2016.10.004},
	Issn = {0021-7824},
	Journal = {Journal de Math{\'e}matiques Pures et Appliqu{\'e}es},
	Pages = {737--757},
	Title = {Pathwise integration with respect to paths of finite quadratic variation},
	Url = {http://www.sciencedirect.com/science/article/pii/S0021782416301155},
	Year = {2017},
	volume = {107},
	number={6}
}

@article{perkowski2019,
  author  = {Cont, Rama and Perkowski, Nicolas},
  title   = {Pathwise integration and change of variable formulas for continuous paths with arbitrary regularity},
  journal = {Transactions of the American Mathematical Society, Series B},
  volume  = {6},
  pages   = {161--186},
  year    = {2019},
  doi     = {10.1090/btran/34}
}

@book{friz2010,
  title={Multidimensional Stochastic Processes as Rough Paths: Theory and Applications},
  author={Friz, P.K. and Victoir, N.B.},
  isbn={9781139487214},
  series={Cambridge Studies in Advanced Mathematics},
  url={https://books.google.co.kr/books?id=CVgwLatxfGsC},
  year={2010},
  publisher={Cambridge University Press}
}

@book{FrizHairer,
	Author = {Friz, Peter K. and Hairer, Martin},
	Doi = {10.1007/978-3-319-08332-2},
	Isbn = {978-3-319-08331-5; 978-3-319-08332-2},
	Mrclass = {60-02 (34F05 35R60 60H07 60H10 60H15)},
	Mrnumber = {3289027},
	Mrreviewer = {Fabrice Baudoin},
	Pages = {xiv+251},
	Publisher = {Springer},
	Series = {Universitext},
	Title = {A course on rough paths},
	Url = {http://dx.doi.org/10.1007/978-3-319-08332-2},
	Year = {2014}}

@incollection{follmer1981,
	Author = {F{\"o}llmer, H.},
	Booktitle = {Seminar on {P}robability, {XV} ({U}niv. {S}trasbourg, {S}trasbourg, 1979/1980) ({F}rench)},
	Mrclass = {60H05},
	Mrnumber = {622559},
	Mrreviewer = {Nicole El Karoui},
	Pages = {143--150},
	Publisher = {Springer, Berlin},
	Series = {Lecture Notes in Math.},
	Title = {Calcul d'{I}t\^o sans probabilit\'es},
	Volume = {850},
	Year = {1981}}

@article{ruhong2022,
  author  = {Cont, Rama and Jin, Ruhong},
  title   = {Fractional {I}to calculus},
  journal = {Transactions of the American Mathematical Society, Series B},
  volume  = {11},
  pages   = {727--761},
  year    = {2024},
  doi     = {10.1090/btran/185}
}
\bibliographystyle{apalike}

\end{document}